\documentclass{article}
\usepackage{amssymb,amsmath,amscd}
\usepackage{amsthm}
\usepackage{latexsym}
\usepackage{color}
\usepackage{enumerate}

\newtheorem{theorem}{Theorem}[section]
\newtheorem{lemma}[theorem]{Lemma}
\newtheorem{corollary}[theorem]{Corollary}
\newtheorem{proposition}[theorem]{Proposition}
\theoremstyle{definition}
\newtheorem{example}[theorem]{Example}
\newtheorem{definition}[theorem]{Definition}
\theoremstyle{remark}
\newtheorem{remark}[theorem]{Remark} 
\numberwithin{equation}{section}
\newcommand{\ignore}[1]{}
\newcommand{\tn}{\ensuremath\mathbf{t}}

\newcommand{\nd}{{\ensuremath d}} 

\newcommand{\beq}{\begin{eqnarray}}
\newcommand{\eeq}{\end{eqnarray}}
\newcommand{\beqq}{\begin{eqnarray*}}
\newcommand{\eeqq}{\end{eqnarray*}}
\newcommand{\benu}{\begin{enumerate}}
\newcommand{\eenu}{\end{enumerate}}

\newcommand{\bit}{\begin{itemize}}
\newcommand{\eit}{\end{itemize}}
\newcommand{\bli}{\begin{list}{}{\labelwidth6mm\leftmargin8mm}}
\newcommand{\eli}{\end{list}}
\newcommand{\bpr}{\begin{proof}}
\newcommand{\epr}{\end{proof}}

\newcommand{\R}{{\mathbb R}}
\newcommand{\real}{\R}
\newcommand{\Rn}{{\mathbb R}^{\nd}} 
\newcommand{\N}{\ensuremath{\mathbb N}}
\newcommand{\No}{\N_{0}}
\newcommand{\nat}{\N}
\newcommand{\no}{\No}
\newcommand{\Z}{\mathbb Z} 
\newcommand{\Zn}{\Z^{\nd}}
\newcommand{\C}{{\mathbb C}}
\newcommand{\ve}{\ensuremath\varepsilon}

\newcommand{\ls}{\lesssim}
\newcommand{\SRn}{\mathcal{S}(\Rn)}
\newcommand{\SpRn}{\mathcal{S}'(\Rn)}

\newcommand{\dint}{\;\mathrm{d}}
\newcommand{\A}{\ensuremath{A^s_{p,q}}}  
\newcommand{\be}{\ensuremath{B^{s_1}_{p_1,q_1}}}  
\newcommand{\bz}{\ensuremath{B^{s_2}_{p_2,q_2}}}

\newcommand{\Ae}{\ensuremath{A^{s_1}_{p_1,q_1}}}  
\newcommand{\Az}{\ensuremath{A^{s_2}_{p_2,q_2}}}
\newcommand{\bmo}{\mathrm{bmo}}
\newcommand{\Lloc}{\ensuremath{L_1^{\mathrm{loc}}}} 
\def\supp{\mathop{\rm supp}\nolimits}
\newcommand{\id}{\ensuremath{\operatorname{id}}}
\newcommand{\Id}{\ensuremath{\operatorname{Id}}}

\newcommand{\whole}[1]{\ensuremath\left\lfloor #1 \right\rfloor}
\newcommand{\M}{\ensuremath{{\mathcal M}_{u,p}}}  
\newcommand{\MA}{\ensuremath{\mathcal A}^{s}_{u,p,q}}  
\newcommand{\MB}{\ensuremath{{\mathcal N}^{s}_{u,p,q}}}  
\newcommand{\at}{{A}_{p,q}^{s,\tau}}
\newcommand{\ate}{{A}_{p_1,q_1}^{s_1,\tau_1}}
\newcommand{\atz}{{A}_{p_2,q_2}^{s_2,\tau_2}}

\newcommand{\fti}{{F}_{p_i,q_i}^{s_i,\tau_i}}

\newcommand{\bt}{{B}_{p,q}^{s,\tau}}
\newcommand{\ft}{{F}_{p,q}^{s,\tau}}
\newcommand{\B}{\ensuremath{B^s_{p,q}}}
\newcommand{\F}{\ensuremath{F^s_{p,q}}}
\newcommand{\cM}{\ensuremath{{\mathcal M}}}
\newcommand{\MBe}{\ensuremath{{\mathcal N}^{s_1}_{u_1,p_1,q_1}}}  
\newcommand{\MBz}{\ensuremath{{\mathcal N}^{s_2}_{u_2,p_2,q_2}}}  
\newcommand{\MAe}{\ensuremath{\mathcal A}^{s_1}_{u_1,p_1,q_1}}  
\newcommand{\MAz}{\ensuremath{{\mathcal A}^{s_2}_{u_2,p_2,q_2}}}  
\newcommand{\MF}{\ensuremath{{\mathcal E}^{s}_{u,p,q}}}  
\newcommand{\MFe}{\ensuremath{{\mathcal E}^{s_1}_{u_1,p_1,q_1}}}  
\newcommand{\MFz}{\ensuremath{{\mathcal E}^{s_2}_{u_2,p_2,q_2}}}  

\newcommand{\MFi}{\ensuremath{{\mathcal E}^{s_i}_{u_i,p_i,q_i}}}

\newcommand{\mbt}{\ensuremath{\widetilde{n}^{\sigma}_{u,p,q}}}  
\newcommand{\mbet}{\ensuremath{\widetilde{n}^{\sigma_1}_{u_1,p_1,q_1}}}  
\newcommand{\mbzt}{\ensuremath{\widetilde{n}^{\sigma_2}_{u_2,p_2,q_2}}}  
\newcommand{\mmb}{\ensuremath{m^{2^{j\nd}}_{u,p}}} 
\newcommand{\mmbet}{\ensuremath{{m}^{2^{j\nd}}_{u_1,p_1}}}  
\newcommand{\mmbzt}{\ensuremath{{m}^{2^{j\nd}}_{u_2,p_2}}}  
\newcommand{\mmbj}{\ensuremath{m^{(j)}_{u,p}}} 
\newcommand{\mmbjet}{\ensuremath{{m}^{(j)}_{u_1,p_1}}}  
\newcommand{\mmbjzt}{\ensuremath{{m}^{(j)}_{u_2,p_2}}}  

\newcommand{\vz}{\ensuremath{\varphi}}

\newcommand{\critical}{\ensuremath{\overline{\gamma}(\tau_1,\tau_2,p_1,p_2)}}
\newcommand{\vr}{\ensuremath{\varrho}}
\newcommand{\rhoA}{\ensuremath\vr\text{-} \!\A}
\newcommand{\rhoAe}{\ensuremath\vr\text{-} \!\Ae}
\newcommand{\rhoAz}{\ensuremath\vr\text{-} \!\Az}

\newcommand{\nn}[1]{\ensuremath \nu( #1)}

\allowdisplaybreaks[1]

\begin{document}

\title{Nuclear embeddings of Morrey sequence spaces \\ and smoothness Morrey spaces}

{\author{Dorothee D. Haroske, Leszek Skrzypczak}
 \date{ }

  \maketitle

\begin{flushright}
  \emph{Dedicated to David E.~Edmunds on the occasion of his 91st birthday}\\
  \emph{and to Hans Triebel on the occasion of his 86th birthday }
\end{flushright}

\smallskip~

  \begin{abstract}
    We study nuclear embeddings for spaces of Morrey type, both in its sequence space version and as smoothness spaces of functions defined on a bounded domain $\Omega\subset\Rn$. This covers, in particular, the meanwhile well-known and completely answered situation for spaces of Besov and Triebel-Lizorkin type defined on bounded domains which has been considered for a long time. The complete result was obtained only recently. Compact embeddings for function spaces of Morrey type have already been studied in detail, also concerning their entropy and approximation numbers. We now prove the first and complete nuclearity result in this context. The concept of nuclearity has  already been introduced by Grothendieck in 1955. Again we rely on suitable wavelet decomposition techniques and the famous Tong result (1969) which characterises nuclear diagonal operators acting between sequence spaces of $\ell_r$ type, $1\leq r\leq\infty$. 
\end{abstract}

\noindent{\em Keywords:} ~ nuclear embeddings, Morrey smoothness spaces, Besov and Triebel-Lizorkin type spaces, Morrey sequence spaces\\
  
\noindent{\bfseries MSC} (2010): 46E35, 47B10 \\

\section{Introduction}

Let $\Omega\subset\Rn$ be a bounded Lipschitz domain and $\MB(\Omega)$ and $\MF(\Omega)$ smoothness Morrey spaces, with $s_i\in\R$, $0<p_i\leq u_i<\infty$, or $p_i=u_i=\infty$, $0<q_i\leq\infty$, $i=1,2$. Roughly speaking, these spaces $\MB$ and $\MF$ are the counterparts of the well-known Besov and Triebel-Lizorkin function spaces $\B$ and $\F$, respectively, where the basic $L_p$ space in the latter scales are replaced by the {Morrey space}
  $\M$, $0<p\le u<\infty $: this is the set of all
  locally $p$-integrable functions $f\in L_p^{\mathrm{loc}}(\Rn)$  such that
\begin{equation}\label{i-Mo}
\|f \mid {\M(\Rn)}\| =\, \sup_{x\in \Rn, R>0} R^{\frac{d}{u}-\frac{d}{p}}
\left[\int_{B(x,R)} |f(y)|^p \dint y \right]^{\frac{1}{p}}\, <\, \infty\, .
\end{equation}
Consequently $\mathcal{M}_{p,p}=L_p$ and likewise $\B=\mathcal{N}^s_{p,p,q}$,  $\F=\mathcal{E}^s_{p,p,q}$. For the precise definition and further properties we refer to Section~\ref{prelim} below.

Motivated also by applications to PDE, smoothness Morrey spaces have been studied  intensely in the last years opening a wide field of possible applications. The Besov-Morrey spaces $\mathcal{N}^{s}_{u,p,q}(\Rn)$ were introduced  in \cite{KY} by Kozono and Yamazaki and used by them and Mazzucato \cite{Maz} to study Navier-Stokes equations. Corresponding Triebel-Lizorkin-Morrey spaces $\mathcal{E}^s_{u,p,q}(\Rn)$ were introduced in \cite{TX} by Tang and Xu, where the authors established the Morrey version of the Fefferman-Stein vector-valued inequality.

Another class of generalisations, the Besov-type space $\bt(\Rn)$ and the Triebel-Lizorkin-type space $\ft(\Rn)$ were introduced in \cite{ysy}. They  coincide with $\B$ and $\F$ when $\tau=0$. Their homogeneous versions were originally investigated by  El Baraka in \cite{ElBaraka1,ElBaraka2, ElBaraka3} and by Yuan and Yang \cite{yy1,yy2}. There are also some applications in partial differential equations for spaces of type $\bt(\Rn)$ and $\ft(\Rn)$, such as (fractional) Navier-Stokes equations, cf. \cite{lxy}.

Although the above scales are defined in different ways, they are closely connected and share a number of properties. Both approaches can be seen as examples of more general scales of Morrey smoothness spaces, we refer to our recent paper \cite{HT6} in this context. Further details, definitions, properties and references can be found in Section~\ref{prelim} below.
 
Parallel to the situation in Besov and Triebel-Lizorkin spaces,   characterisations are well-known, when embeddings between spaces in these scales of Morrey smoothness spaces on $\Rn$ are continuous. But there cannot exist a compact embedding, unless one imposes further assumptions, like weights, certain subspaces, or -- as in the case dealt with in the present paper -- one considers spaces defined on bounded domains. 

In our papers \cite{hs12b,hs14,HaSk-krakow,HaSk-morrey-comp,ghs20} we obtained a complete characterisation of the compactness of the embeddings
\begin{equation}\label{i-0}
\id_{\mathcal{N}}: \MBe(\Omega)\to \MBz(\Omega)\quad \text{and}\quad \id_{\mathcal{E}}:\MFe(\Omega)\to \MFz(\Omega),
\end{equation}
similarly for spaces of type $\bt(\Omega)$ and $\ft(\Omega)$, and we also studied the `degree' of that compactness in terms of entropy and approximation numbers. More precisely, assume that $s_i\in \R$,   $0<p_i\le u_i<\infty$ or $p_i=u_i=\infty$, $i=1,2$, and  $0<q_1,q_2\le \infty$. Then the  embeddings in \eqref{i-0} are compact if, and only if, 
\[\frac{s_1-s_2}{\nd} > \frac{1}{u_1}-\frac{1}{u_2}+\Big(\frac{1}{u_2}-\frac{1}{\max\{1,p_2/p_1\}u_1}\Big)_+ .\]

The main purpose of the present paper is to study the nuclearity of embeddings of type \eqref{i-0}. Grothendieck introduced the concept of nuclearity in \cite{grothendieck} more than 60 years ago. It provided the basis for many famous developments in functional analysis afterwards. Recall that Enflo used nuclearity in his famous solution \cite{enflo} of the approximation problem, a long-standing problem of Banach from the Scottish Book. We refer to \cite{Pie-snumb,Pie-op-2}, and, in particular, to \cite{pie-history} for further historic details.

Let $X,Y$ be Banach spaces, $T\in \mathcal{L}(X,Y)$ a linear and bounded operator. Then $T$ is called {\em nuclear}, denoted by $T\in\mathcal{N}(X,Y)$, 
if there exist elements $a_j\in X'$, the dual space of $X$, and $y_j\in Y$, $j\in\mathbb{N}$, such that $\sum_{j=1}^\infty \|a_j\|_{X'} \|y_j\|_Y < \infty$ 
and a nuclear representation $Tx=\sum_{j=1}^\infty a_j(x) y_j$ for any $x\in X$. Together with the {\em nuclear norm}
\[
 \nn{T}=\inf\Big\{ \sum_{j=1}^\infty   \|a_j\|_{X'} \|y_j\|_Y:\ T =\sum_{j=1}^\infty a_j(\cdot) y_j\Big\},
  \]
  where the infimum is taken over all nuclear representations of $T$, the space $\mathcal{N}(X,Y)$ becomes a Banach space. It is obvious that 
	nuclear operators are, in particular, compact.
	
  Already in the early years there was a strong interest to study examples of nuclear operators beyond diagonal operators in $\ell_p$ 
	sequence spaces, where a complete answer was obtained in \cite{tong}.
	Concentrating on embedding operators in spaces of Sobolev type, first results can be found, for instance, in \cite{PiTri,Pie-r-nuc}. We noticed an increased interest in studies of nuclearity in the last years.  Dealing with the Sobolev embedding for spaces on a bounded domain, some of the recent papers we 
have in mind are \cite{EL-4,EGL-3,Tri-nuclear,CoDoKu,CoEdKu} using quite different techniques however. 

There might be several reasons for this. For example, the problem to describe a compact operator outside the Hilbert space setting is a partly
 open and very important one. 
 It is well known from the remarkable Enflo result \cite{enflo} that there are compact operators
 between Banach spaces which cannot be approximated by finite-rank operators.
 This led to a number of -- meanwhile well-established  and famous -- methods to circumvent this difficulty and find alternative ways to `measure' the compactness or `degree' of compactness of an operator. It can be described by the asymptotic behaviour of its approximation or entropy numbers, 
which are basic tools for many different problems nowadays, e.g. eigenvalue distribution of compact operators in Banach spaces, 
optimal approximation of Sobolev-type embeddings, but also for numerical questions. In all these problems, the decomposition
of a given compact operator into a series is an essential proof technique. It turns out that in many of the recent contributions \cite{Tri-nuclear,CoDoKu,CoEdKu} studying nuclearity, a key tool in the arguments are new decomposition techniques as well, adapted to the different spaces. Inspired by the nice paper \cite{CoDoKu} we also used such arguments in our papers \cite{HaSk-nuc-weight,HaLeoSk}, and intend to follow this strategy here again.

As mentioned above, function spaces of Besov or Triebel-Lizorkin type, as well as their Morrey counterparts, defined on $\Rn$ never admit a compact, let alone nuclear embedding. But replacing $\Rn$ by a bounded Lipschitz domain $\Omega\subset\Rn$, then the question of nuclearity in the scale of Besov and Triebel-Lizorkin spaces has already been solved, cf. \cite{Pie-r-nuc} (with a forerunner in \cite{PiTri}) for the sufficient conditions, and 
 \cite{Tri-nuclear} with some forerunner in \cite{Pie-r-nuc} and partial results in \cite{EGL-3,EL-4} for the necessity of the conditions.  More precisely, for Besov spaces on bounded Lipschitz domains, $B^s_{p,q}(\Omega)$, it is well known that
\[
\id_\Omega^B : \be(\Omega) \hookrightarrow \bz(\Omega)\] 
\text{is nuclear if, and only if,} 
\[    s_1-s_2 > \nd-\nd \max\left(\frac{1}{p_2}-\frac{1}{p_1},0\right),
\]
where $1\leq p_i,q_i\leq \infty$, $s_i\in\real$, $i=1,2$. The counterpart for spaces of type $\MB(\Omega)$, reads now as follows, see Theorem~\ref{comp-NE} below: let  $s_i\in \R$, $1\leq q_i\leq\infty$, $1\leq p_i\leq u_i<\infty$, or $p_i=u_i=\infty$, $i=1,2$.  Then the embedding 
$$ \id_{\mathcal N}: \MBe(\Omega)\hookrightarrow \MBz(\Omega) $$
is nuclear if,  and only if,
$$
\frac{s_1-s_2}{\nd} > 
\frac{1}{u_1}-\frac{1}{u_2}+\frac{1}{\tn(u_1,\max(1,\frac{p_1}{p_2})u_2)},
$$
where $\tn(r_1,r_2)$ is defined via
$$ \frac{1}{\tn(r_1,r_2)} = \begin{cases}
    1, & \text{if}\ 1\leq r_2\leq r_1\leq \infty, \\
    1-\frac{1}{r_1}+\frac{1}{r_2}, & \text{if}\ 1\leq r_1\leq r_2\leq \infty.
  \end{cases}
  $$
  Clearly the two above-mentioned results coincide in case of $u_i=p_i$, $i=1,2$. We obtained parallel results in the context of spaces $\MF(\Omega)$, $\bt(\Omega)$ and $\ft(\Omega)$, see Theorems~\ref{comp-NE} and \ref{C-nuc-at} below.

As already indicated, we follow here the general ideas presented in \cite{CoDoKu} which use decomposition techniques and benefit from  Tong's result \cite{tong} about nuclear diagonal operators acting in sequence spaces of type $\ell_p$. For that reason we study first appropriate Morrey sequence spaces which are adapted to the wavelet decomposition of our Morrey function spaces. The nuclearity result in this sequence space setting can be found in Theorem~\ref{comp-s} below. Here we also rely on our earlier results in \cite{HaSk_Mo_seq}. 

Finally, the present paper can also be seen as an answer to a question by our friend and colleague, David E. Edmunds (University of Sussex at Brighton), who asked us some years ago what is known about nuclear embeddings in the setting of Morrey smoothness spaces. We did not know any result, but became sufficiently fascinated by the topic to study this question ourselves. This is the outcome.

The paper is organised as follows. In Section~\ref{prelim} we recall basic facts 
about the sequence and function spaces we shall work with, Section~\ref{sect-1} is devoted to the general concept of nuclear embeddings and some preceding results. In Section~\ref{nuc-morrey-seq} we concentrate on the Morrey sequence spaces with the main nuclearity result in Theorem~\ref{comp-s}, while Section~\ref{nuc-morrey-func} deals with the question of nuclear embeddings in Morrey smoothness spaces. Here the main findings are collected in Theorems~\ref{comp-NE} and \ref{C-nuc-at}. We conclude the paper with a number of examples.
        }


\section{Function spaces of Morrey type}~\label{prelim}
First we fix some notation. By $\N$ we denote the \emph{set of natural numbers},
by $\N_0$ the set $\N \cup \{0\}$,  and by $\Zn$ the \emph{set of all lattice points
	in $\Rn$ having integer components}.
For $a\in\real$, let   $\whole{a}:=\max\{k\in\Z: k\leq a\}$ and $a_+:=\max\{a,0\}$.
All unimportant positive constants will be denoted by $C$, occasionally with
subscripts. By the notation $A \ls B$, we mean that there exists a positive constant $C$ such that
$A \le C \,B$, whereas  the symbol $A \sim B$ stands for $A \ls B \ls A$.
We denote by $B(x,r) :=  \{y\in \Rn: |x-y|<r\}$ the ball centred at $x\in\Rn$ with radius $r>0$, and $|\cdot|$ denotes the Lebesgue measure when applied to measurable subsets of $\Rn$.

Given two (quasi-)Banach spaces $X$ and $Y$, we write $X\hookrightarrow Y$
if $X\subset Y$ and the natural embedding of $X$ into $Y$ is continuous.

\subsection{Smoothness spaces of Morrey type on $\Rn$}

Let $\SRn$ be the set of all \emph{Schwartz functions} on $\Rn$, endowed
with the usual topology,
and denote by $\SpRn$ its \emph{topological dual}, namely,
the space of all bounded linear functionals on $\SRn$
endowed with the weak $\ast$-topology.
For all $f\in \SRn$ or $\SpRn$, we
use $\widehat{f}$ to denote its \emph{Fourier transform}, and $f^\vee$ for its inverse.
Let $\mathcal{Q}$ be the collection of all \emph{dyadic cubes} in $\Rn$, namely,
$
\mathcal{Q}:= \{Q_{j,k}:= 2^{-j}([0,1)^\nd+k):\ j\in\Z,\ k\in\Zn\}.
$
The {symbol}  $\ell(Q)$ denotes
the side-length of the cube $Q$ and $j_Q:=-\log_2\ell(Q)$.\\

Let $\vz_0,$ $\vz\in\SRn$. We say that $(\vz_0, \vz)$ is an \textit{admissible pair} if
\begin{equation}\label{e1.0}
\supp \widehat{\vz_0}\subset \{\xi\in\Rn:\,|\xi|\le2\}\, , \qquad
|\widehat{\vz_0}(\xi)|\ge C\ \text{if}\ |\xi|\le 5/3,
\end{equation}
and
\begin{equation}\label{e1.1}
\supp \widehat{\vz}\subset \{\xi\in\Rn: 1/2\le|\xi|\le2\}\quad\text{and}\quad
|\widehat{\vz}(\xi)|\ge C\ \text{if}\  3/5\le|\xi|\le 5/3,
\end{equation}
where $C$ is a positive constant.
In what follows, for all $\vz\in\SRn$ and $j\in\N$, $\vz_j(\cdot):=2^{j\nd}\vz(2^j\cdot)$.\\

\begin{definition}\label{d1}
Let $s\in\real$, $\tau\in[0,\infty)$, $q \in(0,\infty]$ and $(\vz_0,\vz)$ be an admissible pair. 
\bli
\item[{\bfseries\upshape (i)}]
Let $p\in(0,\infty]$. The \emph{Besov-type space} $\bt(\Rn)$ is defined to be the collection of all $f\in\SpRn$ such that

$$\|f \mid {\bt(\Rn)}\|:=
\sup_{P\in\mathcal{Q}}\frac1{|P|^{\tau}}\left\{\sum_{j=\max\{j_P,0\}}^\infty\!\!
2^{js q}\left[\int\limits_P
|\vz_j\ast f(x)|^p\dint x\right]^{\frac{q}{p}}\right\}^{\frac1q}<\infty$$

with the usual modifications made in case of $p=\infty$ and/or $q=\infty$.
\item[{\bfseries\upshape (ii)}]
Let $p\in(0,\infty)$. The \emph{Triebel-Lizorkin-type space} $\ft(\Rn)$ is defined to be the collection of all $f\in \SpRn$ such that
$$\|f \mid {\ft(\Rn)}\|:=
\sup_{P\in\mathcal{Q}}\frac1{|P|^{\tau}}\left\{\int\limits_P\left[\sum_{j=\max\{j_P,0\}}^\infty\!\!
2^{js q}
|\vz_j\ast f(x)|^q\right]^{\frac{p}{q}}\dint x\right\}^{\frac1p}<\infty$$
with the usual modification made in case of $q=\infty$.
\eli
\end{definition}

\begin{remark}\label{Rem-Ftau}
	These spaces were introduced in \cite{ysy} and proved therein to be quasi-Banach spaces. In the Banach case  the scale of Nikol'skij-Besov type spaces ${\bt(\Rn)}$ had already been introduced and investigated in \cite{ElBaraka1,ElBaraka2, ElBaraka3}.  It is easy to see that, when $\tau=0$, then $\bt(\Rn)$ and $\ft(\Rn)$
	coincide with the classical
	Besov space $\B(\Rn)$ and Triebel-Lizorkin space $\F(\Rn)$,
	respectively.  There exists extensive literature on such spaces; we
	refer, in particular, to the series of monographs \cite{T-F1,T-F2,t06} for a	comprehensive treatment. In case of $\tau<0$ the spaces are not very interesting, $\bt(\Rn)=\ft(\Rn)=\{0\}$, $\tau<0$.
\end{remark}

\noindent{\em Convention.}~We adopt the nowadays usual custom to write $\A(\Rn)$ instead of $\B(\Rn)$ or $\F(\Rn)$, and $\at(\Rn)$ instead of $\bt(\Rn)$ or $\ft(\Rn)$, respectively, when both scales of spaces are meant simultaneously in some context.

We have elementary embeddings within this scale of spaces (see \cite[Proposition~2.1]{ysy}),
\begin{equation} \label{elem-0-t}
A^{s+\ve,\tau}_{p,r}(\Rn) \hookrightarrow \at(\Rn) \qquad\text{if}\quad \varepsilon\in(0,\infty), \quad r,\,q\in(0,\infty],
\end{equation}
and
\begin{equation} \label{elem-1-t}
{A}^{s,\tau}_{p,q_1}(\Rn)  \hookrightarrow {A}^{s,\tau}_{p,q_2}(\Rn)\quad\text{if} \quad q_1\le q_2,
\end{equation}
as well as
\begin{equation}\label{elem-tau}
B^{s,\tau}_{p,\min\{p,q\}}(\Rn)\, \hookrightarrow \, \ft(\Rn)\, \hookrightarrow \, B^{s,\tau}_{p,\max\{p,q\}}(\Rn),
\end{equation}
which directly extends the well-known classical case from $\tau=0$ to $\tau\in [0,\infty)$, $p\in(0,\infty)$, $q\in(0,\infty]$ and $s\in\real$.

It is also known from \cite[Proposition 2.6]{ysy}
that 
\begin{equation} \label{010319}
A^{s,\tau}_{p,q}(\Rn) \hookrightarrow B^{s+\nd(\tau-\frac1p)}_{\infty,\infty}(\Rn). 
\end{equation} 
The following  remarkable feature was proved in \cite{yy02}.
\begin{proposition}\label{yy02}
	Let $s\in\real$, $\tau\in[0,\infty)$  and $p,\,q\in(0,\infty]$ (with $p<\infty$ in the $F$-case). If
	either $\tau>\frac1p$ or $\tau=\frac1p$ and $q=\infty$, then $A^{s,\tau}_{p,q}(\Rn) = B^{s+\nd(\tau-\frac1p)}_{\infty,\infty}(\Rn)$.
\end{proposition}

Now we come to smoothness  spaces of Morrey type   $\MB(\Rn)$ and $\MF(\Rn)$. The \emph{Morrey space}
$\M(\Rn)$, $0<p\le u<\infty $, is defined to be the set of all
locally $p$-integrable functions $f\in L_p^{\mathrm{loc}}(\Rn)$  such that
$$
\|f \mid {\M(\Rn)}\| :=\, \sup_{x\in \Rn, R>0} R^{\frac{d}{u}-\frac{d}{p}}
\left[\int_{B(x,R)} |f(y)|^p \dint y \right]^{\frac{1}{p}}\, <\, \infty\, .
$$

\begin{remark}
	The spaces $\M(\Rn)$ are quasi-Banach spaces (Banach spaces for $p \ge 1$).
	They originated from Morrey's study on PDE (see \cite{Mor}) and are part of the wider class of Morrey-Campanato spaces; cf. \cite{Pee}. They can be considered as a complement to $L_p$ spaces. As a matter of fact, $\cM_{p,p}(\Rn) = L_p(\Rn)$ with $p\in(0,\infty)$.
	To extend this relation, we put  $\cM_{\infty,\infty}(\Rn)  = L_\infty(\Rn)$. One can easily see that $\M(\Rn)=\{0\}$ for $u<p$, and that for  $0<p_2 \le p_1 \le u < \infty$,
	\begin{equation} \label{LinM}
	L_u(\Rn)= \cM_{u,u}(\Rn) \hookrightarrow  \cM_{u,p_1}(\Rn)\hookrightarrow  \cM_{u,p_2}(\Rn).
	\end{equation}
	In an analogous way, one can define the spaces $\cM_{\infty,p}(\Rn)$, $p\in(0, \infty)$, but using the Lebesgue differentiation theorem, one can easily prove  that
	$\cM_{\infty, p}(\Rn) = L_\infty(\Rn)$.
We refer to the recent monographs \cite{FHS-MS-1,FHS-MS-2} for a detailed treatment and, in particular, application of the concept in the study of PDE.
\end{remark}

\begin{definition}\label{d2.5}
	Let $0 <p\leq  u<\infty$ or $p=u=\infty$. Let  $q\in(0,\infty]$, $s\in \real$ and $\vz_0$, $\vz\in\SRn$
	be as in \eqref{e1.0} and \eqref{e1.1}, respectively.
	\bli
	\item[{\upshape\bfseries (i)}]
	The  {\em Besov-Morrey   space}
	$\MB(\Rn)$ is defined to be the set of all distributions $f\in \SpRn$ such that
	\begin{align}\label{BM}
	\big\|f\mid \MB(\Rn)\big\|=
	\bigg[\sum_{j=0}^{\infty}2^{jsq}\big\| \varphi_j \ast f\mid
	\M(\Rn)\big\|^q \bigg]^{1/q} < \infty
	\end{align}
	with the usual modification made in case of $q=\infty$.
	\item[{\upshape\bfseries  (ii)}]
	Let $u\in(0,\infty)$. The  {\em Triebel-Lizorkin-Morrey  space} $\MF(\Rn)$
	is defined to be the set of all distributions $f\in   \SpRn$ such that
	\begin{align}\label{FM}
	\big\|f \mid \MF(\Rn)\big\|=\bigg\|\bigg[\sum_{j=0}^{\infty}2^{jsq} |
	(\varphi_j\ast f)(\cdot)|^q\bigg]^{1/q}
	\mid \M(\Rn)\bigg\| <\infty
	\end{align}
	with the usual modification made in case of  $q=\infty$.
	\eli
\end{definition}

\noindent{\em Convention.} Again we adopt the usual custom to write $\MA$ instead of $\MB$ or $\MF$, when both scales of spaces are meant simultaneously in some context. 

\begin{remark}
	The  spaces $\MA(\Rn)$ are
	independent of the particular choices of $\vz_0$, $\vz$ appearing in their definitions.
	They are quasi-Banach spaces
	(Banach spaces for $p,\,q\geq 1$), and $\mathcal{S}(\Rn) \hookrightarrow
	\MA(\Rn)\hookrightarrow \mathcal{S}'(\Rn)$.  Moreover, for $u=p$
	we re-obtain the usual  Besov and Triebel-Lizorkin spaces,
	\begin{equation}
	{\mathcal A}^{s}_{p,p,q}(\Rn) = \A(\Rn) = A^{s,0}_{p,q}(\Rn). \label{MB=B}
	\end{equation}
	Besov-Morrey spaces were introduced by Kozono and Yamazaki in
	\cite{KY}. They studied semi-linear heat equations and Navier-Stokes
	equations with initial data belonging to  Besov-Morrey spaces.  The
	investigations were continued by Mazzucato \cite{Maz}, where one can find the
	atomic decomposition of some spaces. The Triebel-Lizorkin-Morrey spaces
	were later introduced by  Tang and Xu \cite{TX}. We follow the
	ideas of Tang and Xu \cite{TX}, where a somewhat  different definition is proposed. The ideas were further developed by Sawano and Tanaka \cite{ST1,ST2,Saw1,Saw2}. The most systematic and general approach to the spaces of this type  can  be found in the book \cite{ysy} or in the survey papers by Sickel \cite{s011,s011a}.  
\end{remark}

It turned out that many of the results from the classical situation have their  counterparts for the spaces $\mathcal{A}^s_{u,p,q}(\Rn)$, e.\,g.,
\begin{equation} \label{elem-0}
{\mathcal A}^{s+\varepsilon}_{u,p,r}(\Rn)  \hookrightarrow
\MA(\Rn)\qquad\text{if}\quad \varepsilon>0, \quad r\in(0,\infty],
\end{equation}
and ${\mathcal A}^{s}_{u,p,q_1}(\Rn)  \hookrightarrow {\mathcal A}^{s}_{u,p,q_2}(\Rn)$ if $q_1\le q_2$. However, there also exist some differences.
Sawano proved in \cite{Saw2} that, for $s\in\real$ and $0<p< u<\infty$,
\begin{equation}\label{elem}
{\mathcal N}^s_{u,p,\min\{p,q\}}(\Rn)\, \hookrightarrow \, \MF(\Rn)\, \hookrightarrow \,{\mathcal N}^s_{u,p,\infty}(\Rn),
\end{equation}
where, for the latter embedding, $r=\infty$ cannot be improved -- unlike
in case of $u=p$ (see \eqref{elem-tau} with $\tau=0$). More precisely,
\[
\MF(\Rn)\hookrightarrow {\mathcal N}^s_{u,p,r}(\Rn)\quad\text{if, and only if,}\quad r=\infty ~~ \text{or} ~~  u=p\ \text{and}\ r\ge \max\{p,\,q\}.
\]
On the other hand, Mazzucato has shown in \cite[Proposition~4.1]{Maz} that
\[
\mathcal{E}^0_{u,p,2}(\Rn)=\M(\Rn),\quad 1<p\leq u<\infty,
\]
in particular,
\begin{equation}\label{E-Lp}
\mathcal{E}^0_{p,p,2}(\Rn)=L_p(\Rn)=F^0_{p,2}(\Rn),\quad p\in(1,\infty).
\end{equation}

\begin{remark}\label{N-Bt-spaces}
	Let $s$, $u$, $p$ and $q$ be as in Definition~\ref{d2.5}
	and $\tau\in[0,\infty)$.
	It is known in \cite[Corollary~3.3, p. 64]{ysy} that
	\begin{equation}
	\MB(\Rn) \hookrightarrow  \bt(\Rn) \qquad \text{with}\qquad \tau={1}/{p}- {1}/{u}.
	\label{N-BT-emb}
	\end{equation}
	Moreover, the above embedding is proper if $\tau>0$ and $q<\infty$. If $\tau=0$ or $q=\infty$, then both spaces coincide with each other, in particular,
	\begin{equation}
	\mathcal{N}^{s}_{u,p,\infty}(\Rn)  =  B^{s,\frac{1}{p}- \frac{1}{u}}_{p,\infty}(\Rn).
	\label{N-BT-equal}
	\end{equation}
	As for the $F$-spaces, if $0\le \tau <{1}/{p}$,
	then
	\begin{equation}\label{fte}
	\ft(\Rn)\, = \, \MF(\Rn)\quad\text{with }\quad \tau =
	{1}/{p}-{1}/{u}\, ,\quad 0 < p\le u < \infty\, ;
	\end{equation}
	cf. \cite[Corollary 3.3, p.\,63]{ysy}. Moreover, if $p\in(0,\infty)$ and $q\in(0,\infty]$, then
	\begin{equation}\label{ftbt}
	F^{s,\, \frac{1}{p} }_{p\, ,\,q}(\Rn) \, = \, F^{s}_{\infty,\,q}(\Rn)\, = \, B^{s,\, \frac1q }_{q\, ,\,q}(\Rn) \, ;
	\end{equation}
	cf. \cite[Propositions 3.4 and 3.5]{s011} and \cite[Remark 10]{s011a}.
\end{remark}

\begin{remark}\label{bmo-def}
Recall that the space $\bmo(\Rn)$ is covered by the above scale. More precisely, consider the local (non-homogeneous) space of functions of bounded mean oscillation, $\bmo(\Rn)$, consisting of all locally integrable
functions $\ f\in \Lloc(\Rn) $ satisfying that
\begin{equation*}
 \left\| f \right\|_{\bmo}:=
\sup_{|Q|\leq 1}\; \frac{1}{|Q|} \int\limits_Q |f(x)-f_Q| \dint x + \sup_{|Q|>
1}\; \frac{1}{|Q|} \int\limits_Q |f(x)| \dint x<\infty,
\end{equation*}
where $ Q $ appearing in the above definition runs over all cubes in $\Rn$, and $ f_Q $ denotes the mean value of $ f $ with
respect to $ Q$, namely, $ f_Q := \frac{1}{|Q|} \;\int_Q f(x)\dint x$,
cf. \cite[2.2.2(viii)]{T-F1}. The space $\bmo(\Rn)$ coincides with $F^{0}_{\infty, 2}(\Rn)$,  cf. \cite[Thm.~2.5.8/2]{T-F1}.  
Hence the above result \eqref{ftbt} implies, in particular,
\begin{equation}\label{ft=bmo}
\bmo(\Rn)= F^{0}_{\infty,2}(\Rn)= F^{0, 1/p}_{p, 2}(\Rn)= {B^{0, 1/2}_{2, 2}(\Rn)}, \quad 0<p<\infty,
\end{equation}
cf. \cite[Propositions~3.4 and 3.5]{s011}.
\end{remark}

\begin{remark}\label{T-hybrid}
In contrast to this approach, Triebel followed the original Morrey-Campanato ideas to develop local spaces $\mathcal{L}^r\A(\Rn)$ in \cite{t13}, and so-called `hybrid' spaces $L^r\A(\Rn)$ in \cite{t14}, where $0<p<\infty$, $0<q\leq\infty$, $s\in\real$, and $-\frac{\nd}{p}\leq r<\infty$. This construction is based on wavelet decompositions and also combines local and global elements as in Definitions~\ref{d1} and \ref{d2.5}. However, Triebel proved in \cite[Thm.~3.38]{t14} that
\begin{equation} \label{hybrid=tau}
L^r\A(\Rn) = \at(\Rn), \qquad \tau=\frac1p+\frac{r}{\nd},
\end{equation}
in all admitted cases. We return to this coincidence below.
\end{remark}

\begin{remark}\label{rem-rho-A}
  The most recent approach to Morrey smoothness spaces can be found in \cite{HT6}: for $-\nd \leq \vr < 0$, $0<p<\infty$, $0<q\leq\infty$, spaces of type $\Lambda^\vr\A(\Rn)$ and $\Lambda_\vr\A(\Rn)$ were introduced there, which satisfy that
  \[
\Lambda_\vr\A(\Rn) = \begin{cases}
  \A(\Rn), & \vr=-\nd, \\  \MA(\Rn), & -\nd\leq \vr<0, \ u\vr + \nd p=0, \end{cases}
        \]
  and 
  \[
    \Lambda^\vr\A(\Rn) = \begin{cases}
\A(\Rn), & \vr=-\nd, \\  \at(\Rn), & \vr\geq -\nd, \ \tau = \frac1p\left(1+\frac{\vr}{\nd}\right). \end{cases}
  \]
  In case of $\vr<-\nd$ one would obtain $\Lambda_\vr\A(\Rn) = \Lambda^\vr\A(\Rn) = \{0\}$, while for $\vr\geq 0$ extensions in case of $\Lambda^\vr\A(\Rn) $ are possible, cf. \cite{HT6}. Moreover, $\Lambda_\vr\F(\Rn)= \Lambda^\vr\F(\Rn)$, $-\nd\leq \vr<0$, whereas 
  $\Lambda_\vr\B(\Rn)\subsetneq \Lambda^\vr\B(\Rn)$ unless $q=\infty$. As many interesting properties of the spaces $\Lambda_\vr\A(\Rn)$ and $\Lambda^\vr\A(\Rn)$ are similar for the same parameter $\vr$, the authors introduced in \cite{HT6} so called $\vr$-clans $\rhoA(\Rn)$ of spaces, $-\nd<\vr<0$, which share such important features. We shall return to this generalisation and concept below. 
\end{remark}

\subsection{Wavelet decomposition}\label{sect-2-2}
\newcommand{\msib}{\ensuremath{n^{\sigma}_{u,p,q}}}  
\newcommand{\msibe}{\ensuremath{n^{\sigma_1}_{u_1,p_1,q_1}}}  
\newcommand{\msibz}{\ensuremath{n^{\sigma_2}_{u_2,p_2,q_2}}} 

We briefly recall the wavelet characterisation of Besov-Morrey  spaces  proved in \cite{Saw2}. It will be essential in our approach.  For $m\in \Zn$ and $\nu \in \Z$ we define  a $\nd$-dimensional dyadic cube   with sides parallel to the axes of coordinates by $ Q_{\nu,m} = \prod_{i=1}^\nd  \left.\left[ \frac{m_i}{2^\nu},\frac{m_i+1}{2^\nu}\right.\right)$, $\nu\in\Z$,  $m=(m_1,\ldots , m_\nd) \in \Zn$. 
For $0<u<\infty$, $\nu \in \Z$ and $m\in\Zn$ we denote by $\chi_{\nu, m}^{(u)}$ the $u$-normalised characteristic function of  the cube $Q_{\nu, m}$, 
$\ \chi_{\nu, m}^{(u)} = 2^{{\nu \nd}/{u}}  \chi_{Q_{\nu,m}}$, 
hence $\| \chi_{\nu, m}^{(u)}|L_p\|=1$ and $\| \chi_{\nu, m}^{(u)}|\M\|=1$. 

Let $\widetilde{\phi}$ be a scaling function  on $\R$ with compact support and of sufficiently high regularity.
Let $\widetilde{\psi}$ be an associated wavelet. Then the  tensor-product ansatz yields a scaling function $\phi$  and associated wavelets
$\psi_1, \ldots, \psi_{2^{d}-1}$, all defined now on $\Rn$.  We suppose $\widetilde{\phi} \in C^{N_1}(\R)$ and $\supp \widetilde{\phi}
\subset [-N_2,\, N_2]$ for certain natural numbers $N_1$ and $N_2$. This implies
\begin{equation}\label{2-1-2}
	\phi, \, \psi_i \in C^{N_1}(\Rn) \quad \text{and} \quad 
	\supp \phi ,\, \supp \psi_i \subset [-N_3,\, N_3]^\nd , 
\end{equation}
for $i=1, \ldots \, , 2^{\nd}-1$. We use the standard abbreviations 
\begin{equation}
	\phi_{\nu,m}(x) =  2^{\nu \nd/2} \, \phi(2^\nu x-m) \quad
	\text{and}\quad
	\psi_{i,\nu,m}(x) =  2^{\nu \nd/2} \, \psi_i(2^\nu x-m) . 
\end{equation}

To formulate the result we  introduce some sequence  spaces. For $0<p\le u<\infty$, or $p=u=\infty$, $0< q \leq\infty$ and  $\sigma\in \R$, let
\begin{multline}
	n^{\sigma}_{u,p,q}  :=   \Bigg\{ \lambda = 
	\{\lambda_{\nu,m}\}_{\nu,m} : \     \lambda_{\nu,m} \in \C\, , 
	\\
	\| \, \lambda \, |n^{\sigma}_{u,p,q}\| = \Big\| 
	\Big\{2^{\nu(\sigma-\frac{\nd}{u}) }\,  \Big\|\sum_{m \in
		\Zn}\lambda_{\nu,m}\, \chi^{(u)}_{\nu,m}| \M 
	\Big\|\Big\}_{\nu\in\no} | \ell_q\Big\| < \infty \Bigg\}\, . 
	\label{mbspqr}
\end{multline}
However, in many  situations the  following equivalent norm in  the space $	n^{\sigma}_{u,p,q}$ is more useful
	\beq \nonumber
\|\lambda|\msib\|^\ast = \Big(\sum_{j=0}^\infty 2^{qj(\sigma-\frac \nd u)}\!\!\!\sup_{\nu: \nu \le j; k\in \Zn}\!\! 2^{q\nd (j-\nu)(\frac 1 u - \frac 1 p )}\big(\!\sum_{m:Q_{j,m}\subset Q_{\nu,k}}\!\!\!\!|\lambda_{j,m}|^p\big)^{\frac q p}\Big)^{\frac 1 q},
\!\!\,
\label{3-0}
\eeq
cf. \cite{HaSk-bm1}. 
The following theorem was proved in \cite{Saw2}. 

\begin{theorem}\label{wavemorrey}  
	Let $0 < p\le u < \infty$ or $u=p=\infty$, $0<q\le \infty$ and let $s\in \R$. 
	Let $\phi$ be a scaling function and let $\psi_i$,  $i=1, \ldots ,2^\nd -1$, be
	the corresponding wavelets satisfying \eqref{2-1-2}. We assume that {$ \max\left\{ (1+\whole{s})_+, \whole{\nd(\frac 1 p -1)_+ -s} \right\}  \le N_1$}. Then a distribution $f \in \SpRn$ belongs to $\MB(\Rn)$,
	if, and only if, 
	\begin{align*}
		\| \, f \, |\MB(\Rn)\|^\star  = \ &  
		\Big\| \left\{\langle f,\phi_{0,m}\rangle \right\}_{m\in \Z^\nd } |
		\ell_u\Big\| 
		+ \sum_{i=1}^{2^\nd -1}
		\Big\| \left\{\langle f,\psi_{i,\nu,m}\rangle \right\}_{\nu\in \N_0, m\in \Z^\nd } | \msib \Big\|
	\end{align*}
	is finite, where $\sigma=s+\frac \nd 2 $. Furthermore,
	$\| \, f \, |\MB(\Rn) \|^\star $ may be used as an 
	equivalent $($quasi-$)$ norm in 
	$\MB(\Rn)$.
\end{theorem} 

\begin{remark}\label{wavemorreyrem}
	It follows from Theorem~\ref{wavemorrey} that the mapping
	\begin{equation}\label{wavemorreyrem1}
		T\,:\,f\;\mapsto \; \Big( \left\{\langle f,\phi_{0,m}\rangle \right\}_{m\in \Z^\nd }, \left\{\langle
		f,\psi_{i,\nu,m}\rangle \right\}_{\nu\in \N_0, m\in \Zn, i=1,\ldots, 2^\nd-1}\Big)
	\end{equation}
	is an isomorphism of  $\MB(\Rn) $ onto $\ell_u\oplus
	\big(\oplus_{i=1}^{2^\nd-1} \msib\big)$, $\sigma=s+\frac \nd 2 $, 
	cf. \cite{Saw2}.
	
	The theorem covers the characterisation of Besov spaces $B^s_{p,q}(\Rn)$ by Daubechies wavelets,  cf. \cite[p.26-34]{t06} and the references given there. 
\end{remark}



\subsection{Spaces on domains}
Let $\Omega$ denote a bounded Lipschitz domain in  $\Rn$. We consider smoothness Morrey spaces on $\Omega$ defined by restriction. Let ${\mathcal D}(\Omega)$ be the set of all infinitely differentiable functions supported in $\Omega$ and denote by ${\mathcal D}'(\Omega)$ its dual.

\begin{definition}\label{tau-spaces-Omega}
	Let $s\in\R$, $0<p\le u <\infty$ or $p=u=\infty$, $0<q\le\infty$.
Then  $\MA(\Omega)$ is defined by
	\[
	\MA(\Omega):=\big\{f\in {\mathcal D}'(\Omega): f=g\vert_{\Omega} \text{ for some } g\in \MA(\Rn)\big\}
	\]
	endowed with the quasi-norm
	\[
	\big\|f\mid \MA(\Omega)\big\|:= \inf \big\{ \|g\mid \MA(\Rn)\|:  f=g\vert_{\Omega}, \; g\in  \MA(\Rn)\big\}.
	\]
\end{definition}

\begin{remark}\label{tau-onOmega}
  The spaces $\at(\Omega)$ are defined in the same way by restriction. They are as well as the spaces $\MA(\Omega)$ quasi-Banach spaces (Banach spaces for $p,q\geq 1$).  When $p=u$ or $\tau=0$ we obtain the usual Besov and Triebel-Lizorkin spaces defined on domains. In \cite{ghs21} we studied the extension operator of spaces $\at(\Omega)$ and studied limiting embeddings. We obtained, for instance, that -- in addition to the monotonicity in the smoothness parameter $s$ and the fine index $q$, as recalled in \eqref{elem-0-t} and \eqref{elem-1-t}, respectively, -- there is some monotonicity in $\tau$, too: we proved that $A^{s,\tau_1}_{p,q}(\Omega) \hookrightarrow A^{s,\tau_2}_{p,q}(\Omega)$ when $0\leq \tau_2\leq \tau_1$, cf.   \cite[Proposition~3.9]{ghs21}.  
\end{remark}

The sufficient and necessary conditions for compactness of embeddings of the Besov-Morrey and Triebel-Lizorkin spaces were proved in \cite{hs12b} and \cite{hs14} with a small contribution in the case $p=u$ in \cite{ghs20}. In the last paper one can also find the corresponding results for Besov-type and Triebel-Lizorkin type spaces. The above conditions read as follows.

\begin{theorem}\label{comp}
  Let $s_i\in \R$,   $0<p_i\le u_i<\infty$ or $p_i=u_i=\infty$, $i=1,2$. Moreover, let $0<q_1,q_2\le \infty$. 
  Then the  embedding 
	\begin{equation}\label{id_O_MA}
        \id_{\mathcal{A}}:  \MAe(\Omega)\hookrightarrow \MAz(\Omega)
        \end{equation}
	is compact if, and only if, 
\[\frac{s_1-s_2}{\nd} > \frac{1}{u_1}-\frac{1}{u_2}+\Big(\frac{1}{u_2}-\frac{1}{\max\{1,p_2/p_1\}u_1}\Big)_+ .\]
\end{theorem}

\begin{remark}
  We refer to \cite{hs12b} for further details. In \cite{HaSk-morrey-comp} we studied entropy numbers of such compact embeddings, see also \cite{HaSk-krakow} for some first results on corresponding approximation numbers.
    \end{remark}

The counterpart for spaces $\at(\Omega)$ can be found in \cite{ghs20}. Let us introduce the notation
\begin{align}
{\gamma(\tau_1,\tau_2,p_1,p_2)} & = \max\left\{\left(\tau_2-\frac{1}{p_2}\right)_+ -\left(\tau_1-\frac{1}{p_1}\right)_+,                                  \frac{1}{p_1}-\tau_1 - \min\left\{\frac{1}{p_2}-\tau_2, \frac{1}{p_2}(1-p_1\tau_1)_+\right\}\right\}\nonumber\\
& = \begin{cases}
      \frac{1}{p_1}-\tau_1-\frac{1}{p_2}+\tau_2, & \text{if}\quad \tau_2\geq \frac{1}{p_2}, \\[1ex]
\frac{1}{p_1}-\tau_1, &\text{if}\quad  \tau_1\geq \frac{1}{p_1}, \ \tau_2< \frac{1}{p_2}, \\[1ex]
     \max\left\{0, \frac{1}{p_1}-\tau_1-\frac{1}{p_2}+\max\left\{\tau_2,\frac{p_1}{p_2}\tau_1\right\}\right\}, &\text{if}\quad  \tau_1< \frac{1}{p_1}, \  \tau_2< \frac{1}{p_2}.
    \end{cases}
\label{gamma}
\end{align}

\begin{theorem}  \label{comp-tau}
  Let  $s_i\in \real$, $0<q_i\leq\infty$, $0<p_i\leq \infty$ (with $p_i<\infty$ in case of $A=F$), $\tau_i\geq 0$, $i=1,2$. 
    The embedding
\begin{equation} \label{tau-comp-u1}
	 \id_{\tau}: \ate(\Omega )\hookrightarrow \atz(\Omega )
\end{equation}
is compact if, and only if,
\begin{equation}\label{tau-comp-u2}
  \frac{s_1-s_2}{\nd} > \gamma(\tau_1,\tau_2,p_1,p_2).
\end{equation}
\end{theorem}

\begin{remark}\label{comp-bmo}
In case of special target spaces, $L_\infty(\Omega)$  and $\bmo(\Omega)$ we thus have the following results, recall Remark~\ref{bmo-def}. Let $s\in\real$ and $0<q\leq\infty$. 
\bli
  \item[{\upshape\bfseries (i)}]
Assume $0<p\le u<\infty$.  Then 
\[
  \id:  \MA(\Omega) \hookrightarrow L_{\infty}(\Omega)\quad\text{compact} \iff \id:  \MA(\Omega) \hookrightarrow \bmo(\Omega)\quad\text{compact} \iff \ s> \frac{\nd}{u}.
  \]
  \item[{\upshape\bfseries (ii)}]
Assume $\tau\geq 0$ and $0< p\leq\infty$ (with $p<\infty$ if $A=F$). Then
    \[
\id: \at(\Omega) \hookrightarrow L_{\infty}(\Omega)\quad\text{compact} \iff
\id: \at(\Omega )\hookrightarrow \bmo(\Omega)\quad\text{compact} \iff \  s> \nd\left(\frac1p-\tau\right),\]
\eli
 cf. \cite[Corollaries~3.7, 3.8]{ghs20}.    
\end{remark}

  \begin{remark}\label{comp-rho}
     We return to the spaces $\rhoA$ introduced in Remark~\ref{rem-rho-A}, $-\nd<\vr<0$, $0<p<\infty$, $0<q\leq\infty$. Their restriction to $\Omega$ is defined by restriction in literally the same way as in Definition~\ref{tau-spaces-Omega}. Then the parallel result to Theorems~\ref{comp} and \ref{comp-tau} reads as follows: the embedding
    \begin{equation}\label{id-Om-rho}
      \id_{\Omega,\vr} : \rhoAe(\Omega) \hookrightarrow \rhoAz(\Omega) 
    \end{equation}
    is compact if, and only if,
    \begin{equation}\label{comp-varrho}
      s_1-s_2 > |\vr| \left(\frac{1}{p_1}-\frac{1}{p_2}\right)_+,
      \end{equation}
where $-\nd<\vr<0$, $0<p_i<\infty$, $0<q_i\leq\infty$, $s_i\in\real$, $i=1,2$, cf. \cite[Theorem~6.5]{HT6}. There are also extensions to different parameters  $\vr_1, \vr_2$, cf. \cite[Theorem~6.19]{HT6}. Recall that $\vr=-\nd$ corresponds to the classical situation, i.e., $\rhoA=\A$ if $\vr=-\nd$. Then condition \eqref{comp-varrho} reads exactly as the well-known necessary and sufficient condition for $\id_\Omega: \Ae(\Omega)\hookrightarrow \Az(\Omega)$ to be compact. We coined for such phenomena in \cite{HT6} the notion {\em Slope-$\nd$-rule}. In other words, the compact embeddings of Morrey smoothness spaces of the above type is just an example of such a {\em Slope-$\nd$-rule}. We shall observe a parallel phenomenon in case of nuclearity later.
\end{remark}


\section{Nuclear operators}\label{sect-1}

\begin{definition}
An operator $T\in \mathcal{L}(X,Y)$ is called nuclear if it can be written in the form
\begin{equation}\label{nuclearrep}
T = \sum_{k=1}^\infty x^*_k\otimes y_k,\qquad \text{such that} \qquad \sum_{k=1}^\infty \|x^*_k\|\, \| y_k\|<\infty,
\end{equation}
where $x_1^*,x_2^*,\ldots \in X'$, $y_1,y_2,\ldots \in Y$  and $x^*_k\otimes y_k:x\mapsto x^*_k(x) y_k$.  The nuclear norm is given by 
\begin{equation} \nu(T) = \inf \sum_{k=1}^\infty \|x^*_k\|\, \| y_k\|,\label{nuclearnorm}
  \end{equation}
where the infimum is taken over all representations \eqref{nuclearrep}
\end{definition}

\begin{remark}
	The concept of nuclear operators has been introduced by Grothendieck \cite{grothendieck} and was intensively studied afterwards, cf. 
	\cite{pie-84,Pie-op-2} and also \cite{pie-history} for some history. \\
          One can easily see that if $T$ is a nuclear operator, then the infinite series of the terms $x^*_k\otimes y_k:x\mapsto x^*_k(x) y_k$ is convergent in $\mathcal{L}(X,Y)$. So any nuclear operator can be approximated by finite-rank operators. It is well-known that $\mathcal{N}(X,Y)$ possesses the ideal property. In Hilbert spaces $H_1,H_2$, the nuclear operators $\mathcal{N}(H_1,H_2)$ coincide with the trace class $S_1(H_1,H_2)$, consisting of those $T$ with singular numbers $(s_n(T))_n \in \ell_1$.
  \end{remark}

We collect some further properties for later use and for convenience.
  
\begin{proposition}\label{coll-nuc}
\benu[\bfseries\upshape (i)]
\item  If $X$ is an $n$-dimensional Banach space, then 
$$\nn{\id:X\rightarrow X}= n. $$ 
\item
  For any Banach spaces $X$ and any  bounded linear operator $T:\ell^n_\infty\rightarrow X$ we have 
\[\nn{T} = \sum_{i=1}^n \|Te_i |X \| .\]
\item
  If $T\in \mathcal{L}(X,Y)$ is a nuclear operator and $S\in \mathcal{L}(X_0,X)$ and $R\in \mathcal{L}(Y,Y_0)$, then $RTS$ is a nuclear operator and 
  \begin{equation}\label{i1}
    \nn{RTS} \le  \|R\|\ \nn{T} \ \|S\|.
    \end{equation} 
\eenu
\end{proposition}

Let us introduce the following notation: for numbers $r_1,r_2\in [1,\infty]$, let $\tn(r_1,r_2)$ be given by 
\begin{equation}\label{tongnumber}
\frac{1}{\tn(r_1,r_2)} = \begin{cases}
    1, & \text{if}\ 1\leq r_2\leq r_1\leq \infty, \\
    1-\frac{1}{r_1}+\frac{1}{r_2}, & \text{if}\ 1\leq r_1\leq r_2\leq \infty.
  \end{cases}
\end{equation}
Hence $1\leq \tn(r_1,r_2)\leq \infty$, and 
\[ \frac{1}{\tn(r_1,r_2)}= 1-\left(\frac{1}{r_1}-\frac{1}{r_2}\right)_+ \geq \frac{1}{r^\ast}= \left(\frac{1}{r_2}-\frac{1}{r_1}\right)_+\ ,\]
with $\tn(r_1,r_2)=r^\ast $ if, and only if, $\{r_1,r_2\}=\{1,\infty\}$.

Recall that $c_0$ denotes the subspace of $\ell_\infty$ containing the null sequences. We heavily rely in our arguments below on the following remarkable result by Tong \cite{tong}.

\begin{proposition}[{\cite[Thms.~4.3, 4.4]{tong}}]\label{prop-tong}
  Let $1\leq r_1,r_2\leq\infty$ and $\tau=(\tau_j)_{j\in\nat}$ be a scalar sequence. Denote by $D_\tau$ the corresponding diagonal operator, $D_\tau: x=(x_j)_j \mapsto (\tau_j x_j)_j$, acting between $\ell_{r_1}$  and $\ell_{r_2}$.
\benu[\bfseries\upshape (i)]
\item
  Then $D_\tau$ is nuclear if, and only if, $\tau=(\tau_j)_j \in \ell_{\tn(r_1,r_2)}$, with $\ell_{\tn(r_1,r_2)}= c_0$ if $\tn(r_1,r_2)=\infty$. Moreover,
  \[
  \nn{D_\tau:\ell_{r_1}\to\ell_{r_2}} = \|\tau|{\ell_{\tn(r_1,r_2)}}\|.
  \]
\item
  Let $n\in\nat$ and $D^n_\tau: \ell^n_{r_1}\to \ell^n_{r_2}$ be the corresponding diagonal operator $D_\tau^n: x=(x_j)_{j=1}^n \mapsto (\tau_j x_j)_{j=1}^n$. Then 
\begin{equation}\label{tong-res-Dt}
\nn{D_\tau^n:\ell_{r_1}^n\rightarrow \ell^n_{r_2}} = \left\| (\tau_j)_{j=1}^n | {\ell_{\tn(r_1,r_2)}^n} \right\|.
\end{equation}
\eenu
\end{proposition}

\begin{example}
In the special case of $\tau\equiv 1$, i.e., $D_\tau=\id$, (i) is not applicable and (ii) reads as  
  \begin{equation}\label{tong-res}
\nn{\id :\ell_{r_1}^n\rightarrow \ell^n_{r_2}} =
\begin{cases} 
n & \text{if}\qquad 1\le r_2\le r_1\le \infty,\\
n^{1-\frac{1}{r_1}+\frac{1}{r_2}} & \text{if}\qquad 1\le r_1\le r_2\le \infty .
\end{cases}
\end{equation}
In particular, $\nn{\id:\ell_1^n\rightarrow \ell^n_\infty}=1$. 
\end{example}

\begin{remark}
We refer also to \cite{Pie-op-2} for the case $r_1=1$, $r_2=\infty$. 
  \end{remark}

\begin{theorem}[{\cite{Tri-nuclear,HaSk-nuc-weight}}]\label{prod-id_Omega-nuc}
  Let $\Omega\subset\Rn$ be a bounded Lipschitz domain, $1\leq p_i,q_i\leq \infty$ (with $p_i<\infty$ in the $F$-case), $s_i\in\real$. Then the embedding
  \begin{equation}\label{id_Omega-nuclear-0}
    \id_\Omega : \Ae(\Omega) \to \Az(\Omega)
\end{equation}
    is nuclear if, and only if,
  \begin{equation}
    s_1-s_2 > \nd-\nd\left(\frac{1}{p_2}-\frac{1}{p_1}\right)_+.
\label{id_Omega-nuclear}
  \end{equation}
\end{theorem}

\begin{remark}\label{R-nuc-dom-class}
  The proposition is stated in \cite{Tri-nuclear} for the $B$-case only, but due to the independence of \eqref{id_Omega-nuclear} of the fine parameters $q_i$, $i=1,2$, and in view of (the corresponding counterpart of) \eqref{elem-tau} (with $\tau=0$) it can be extended immediately to $F$-spaces. The if-part of the above result is essentially covered by \cite{Pie-r-nuc} (with a forerunner in \cite{PiTri}). Also part of the necessity of \eqref{id_Omega-nuclear} for the nuclearity of $\id_\Omega$ was proved by Pietsch in \cite{Pie-r-nuc} such that only the limiting case $ s_1-s_2 = \nd-\nd(\frac{1}{p_2}-\frac{1}{p_1})_+$ was open for many decades. Edmunds, Gurka and Lang in \cite{EGL-3} (with a forerunner in \cite{EL-4}) obtained some answer in the limiting case which was then completely solved in \cite{Tri-nuclear}. 
  Note that in \cite{Pie-r-nuc} some endpoint cases (with $p_i,q_i\in \{1,\infty\}$) were already discussed for embeddings of Sobolev and certain Besov spaces (with $p=q$) into Lebesgue spaces. In our paper \cite{HaSk-nuc-weight} we were able to further extend Theorem~\ref{prod-id_Omega-nuc} in view of the borderline cases. Here we essentially benefited from the strategy of the proof presented in \cite{CoDoKu} which studies nuclear embeddings of spaces with modified smoothness.

  For better comparison one can reformulate the compactness and nuclearity characterisations of $\id_\Omega$ in \eqref{id_Omega-nuclear-0} as follows, involving the number $\tn(p_1,p_2)$ defined in \eqref{tongnumber}. Let $1\leq p_i,q_i\leq \infty$, $s_i\in\real$ and 
  \[\delta= s_1 - \frac{\nd}{p_1}-s_2 + \frac{\nd}{p_2}.\]
   Then
  \begin{align*}
    \id_\Omega: \Ae(\Omega) \to \Az(\Omega) \quad  \text{is compact}\quad & \iff \quad \delta> \frac{\nd}{p^\ast}\qquad\text{and}\\
   \id_\Omega: \Ae(\Omega) \to \Az(\Omega) \quad \text{is nuclear}\quad & \iff \quad \delta > \frac{\nd}{\tn(p_1,p_2)}.
  \end{align*}
  Hence apart from the extremal cases $\{p_1,p_2\}=\{1,\infty\}$ (when $\tn(p_1,p_2)=p^\ast$) nuclearity is indeed stronger than compactness.
    We observed similar phenomena -- including the replacement of $p^\ast$ and $q^\ast$ (for compactness assertions) by $\tn(p_1,p_2)$, $\tn(q_1,q_2)$ (for their nuclearity counterparts) -- in the weighted setting in \cite{HaSk-nuc-weight} as well as for vector-valued sequence spaces and function spaces on quasi-bounded domains in \cite{HaLeoSk}.
\end{remark}

\section{Nuclear embedding of Morrey sequence spaces}\label{nuc-morrey-seq}

Our main goal is to find a counterpart of Theorem~\ref{prod-id_Omega-nuc} when $\id_\Omega$ in \eqref{id_Omega-nuclear-0} is replaced by $\id_{\mathcal{A}}$ in \eqref{id_O_MA} or $\id_\tau$ in \eqref{tau-comp-u1}, respectively. We follow the strategy introduced in \cite{CoDoKu} and use wavelet decomposition arguments, based on Section~\ref{sect-2-2}, together with related nuclearity results for appropriate sequence spaces, extending Proposition~\ref{prop-tong} to our setting. 

\subsection{Finite-dimensional Morrey sequence spaces}\label{nuc-fin-morrey}

First we deal with finite-dimensional sequence spaces of Morrey type. 

\begin{definition}
Let $0<p\leq u\le\infty$, $j\in\no$ be fixed and $\mathcal{K}_j = \{k\in\Zn: Q_{0,k}\subset Q_{-j,0}\}$ . We define
\begin{align}
\mmb = & \{ \lambda = \{\lambda_{k}\}_{k\in \mathcal{K}_j} \subset\C: \nonumber\\
&\quad  \|\lambda|\mmb\| = \sup_{Q_{-\nu,m}\subset Q_{-j,0}}\!\!
|Q_{-\nu,m}|^{\frac1u-\frac1p} \Big(\sum_{k:Q_{0,k}\subset Q_{-\nu,m}}\!\!|\lambda_k|^p\Big)^{\frac 1 p}<\infty\},
\label{mseq-defi}
\end{align}
where the  supremum is taken over all $\nu\in\no$ and $m\in\Zn$ such that 
$Q_{-\nu,m}\subset Q_{-j,0}$.
\end{definition}

\begin{remark}
Similarly one can define  spaces related to any cube $Q_{-j,m}$, $m\in \Z^\nd $, but they are isometrically isomorphic to $\mmb$, so we restrict our attention to the last space. \\
Clearly, for $u=p$ this space coincides with the usual $2^{j\nd }$-dimensional space $\ell_p^{2^{j\nd }}$, that is, $m_{p,p}^{2^{j\nd }} = \ell_p^{2^{j\nd }}$. Moreover  $m_{\infty,p}^{2^{j\nd }}= \ell_\infty^{2^{j\nd }}$ for any $p\le \infty$, cf. \cite{HaSk_Mo_seq}. In the sequel it will be convenient to denote this  space by $m_{\infty,\infty}^{2^{j\nd }}$.

\end{remark}

\begin{lemma}\label{lemma15030}
Let $0< p_1\le u_1<\infty$,  $0< p_2\le u_2<\infty$, and $j\in \no$ be given. Then the norm of the compact identity operator  
\begin{equation}\label{id_j-m}
 \id_j: \mmbet\hookrightarrow \mmbzt
\end{equation}
satisfies
\begin{equation}\label{1503-0}
\|\id_j: \mmbet\to \mmbzt\| = 
\begin{cases}
1 &\qquad \text{if} \quad p_1\ge p_2\quad \text{and }\quad u_2\ge u_1,\quad  \\
 1 &\qquad \text{if} \quad p_1 < p_2\quad \text{and }\quad \frac{p_2}{u_2} \le \frac{p_1}{u_1},\\  
 2^{j\nd (\frac{1}{u_2}-\frac{1}{u_1})}&\qquad \text{if} \quad p_1\ge p_2\quad \text{and }\quad u_2 < u_1,\\
\end{cases}
\end{equation}
and in the remaining case, there is a constant $c$, $0<c\le 1$, independent of $j$, such that 
\begin{equation}\label{1503-a}
c\, 2^{j\nd (\frac{1}{u_2}-\frac{p_1}{u_1p_2})} \le\|\id_j: \mmbet\to \mmbzt\| \leq 
2^{j\nd (\frac{1}{u_2}-\frac{p_1}{u_1p_2})}  \qquad \text{if} \quad  
p_1 < p_2\quad \text{and }\quad \frac{p_2}{u_2} > \frac{p_1}{u_1}\ .
\end{equation}
\end{lemma}

\begin{remark}
The above result can be found in \cite{HaSk_Mo_seq}.
\end{remark}

\begin{corollary}\label{cor15030}
	Let $0< p\le u<\infty$   and $j\in \no$ be given. Then the norm of the  identity operator  
	\begin{equation}\label{id_j-uinfty}
		\id_j: m^{2^{j\nd }}_{u,p}\hookrightarrow m^{2^{j\nd }}_{\infty,\infty}
	\end{equation}
 equals $1= \|\id_j: m^{2^{j\nd }}_{u,p}\to m^{2^{j\nd }}_{\infty,\infty} \|$, whereas  the norm of the operator
 \begin{equation}\label{id_j-inftyu}
 	\id_j: m^{2^{j\nd }}_{\infty,\infty}\hookrightarrow m^{2^{j\nd }}_{u,p}
 \end{equation}
 satisfies  $\|\id_j: m^{2^{j\nd }}_{\infty,\infty}\to m^{2^{j\nd }}_{u,p}\|= 2^{j\nd /u}$. 
\end{corollary}

\begin{proof}
The value of the norm of the operator \eqref{id_j-uinfty} follows directly  from the definition of the spaces. The upper  estimate of the norm of the second operator can be proved in a similar way. The estimate from below follows from \eqref{1503-0}  and the simple factorisation
$ m^{2^{j\nd }}_{u_1,p_1}\hookrightarrow m^{2^{j\nd }}_{\infty,\infty}\hookrightarrow m^{2^{j\nd }}_{u,p}$ \quad \text{with} $\ u_1>p_1>u>p$.    
\end{proof}

Now we can give its counterpart for the nuclear norm $\nn{\id_j}$ which also extends Tong's result,  Proposition~\ref{coll-nuc}(ii), from spaces $\ell_p^{2^{j\nd }}$ to $\mmb$.

\begin{proposition}\label{nuclear-m}
  Let $1\leq p_i\leq u_i<\infty$, or $p_i=u_i=\infty$, $i=1,2$, $j\in\no$,  and $\id_j$ be given by \eqref{id_j-m}. Then the nuclear norm of $\id_j$ satisfies
  \begin{equation}\label{nu-norm-m-1}
\nn{\id_j} = 
\begin{cases}
  2^{j\nd } &\qquad \text{if}  \quad p_1\le p_2 \quad\text{and}\quad u_2\le u_1,\\
  &\qquad \text{or}\quad  p_1> p_2 \quad\text{and}\quad\frac{p_1}{u_1} \le \frac{p_2}{u_2}, \\
  &\qquad \text{or}\quad  p_1=u_1=\infty, \\
2^{j\nd (1-\frac{1}{u_1}+\frac{1}{u_2}) }&\qquad \text{if} \quad p_1\le p_2\quad \text{and}\quad u_1< u_2,\\
\end{cases}
\end{equation}
and in the remaining case, there is a constant $c\geq 1$, independent of $j$, such that 
\begin{equation}\label{nu-norm-m-2}
2^{j\nd (1+\frac{p_2}{p_1u_2}-\frac{1}{u_1})} \leq \nn{\id_j} \leq\ c\  2^{j\nd (1+\frac{p_2}{p_1u_2}-\frac{1}{u_1})} \quad \text{if} \quad  
p_1 > p_2\quad \text{and }\quad \frac{p_1}{u_1} > \frac{p_2}{u_2}\ .
\end{equation}
\end{proposition}

\begin{proof}
{\em Step 1}. First we deal with all the estimates from below and benefit from Lemma~\ref{lemma15030} and Corollary \ref{cor15030}. Note that by Proposition~\ref{coll-nuc}(i) and (iii),
\begin{align}
2^{j\nd }= & \, \nn{\id_j: \mmbet\hookrightarrow \mmbet}  \nonumber\\
 \le & \,\nn{\id_j: \mmbet\hookrightarrow \mmbzt} \| \id_j: \mmbzt\hookrightarrow \mmbet\|. 
\label{est-below}
\end{align}
Now in the first three cases of \eqref{nu-norm-m-1} we have $\| \id_j: \mmbzt\hookrightarrow \mmbet\| =1$, such that \eqref{est-below} leads to $\nn{\id_j}\geq 2^{j\nd }$ as desired. In the last case of \eqref{nu-norm-m-1}, 
Lemma~\ref{lemma15030} and Corollary~\ref{cor15030} provide $\| \id_j: \mmbzt\hookrightarrow \mmbet\| =2^{j\nd (\frac{1}{u_1}-\frac{1}{u_2})}$ and thus \eqref{est-below} implies $\nn{\id_j} \geq        2^{j\nd (1-\frac{1}{u_1}+\frac{1}{u_2}) }$. Finally, in case of \eqref{nu-norm-m-2}, then
\[\| \id_j: \mmbzt\hookrightarrow \mmbet\| \leq 2^{j\nd (\frac{1}{u_1}-\frac{p_2}{p_1u_2})}\]
in view of Lemma~\ref{lemma15030} again, especially \eqref{1503-a}, which together with \eqref{est-below} completes the lower estimate in that case.
 \\

{\em Step 2.}\quad Now we show the estimates from above in the first three cases of \eqref{nu-norm-m-1}. We have the following commutative diagram
\[
\begin{CD}
\mmbet@>\id_j>> \mmbzt\\
@V{\id_1}VV @AA{\id_2}A\\
\ell^{2^{j\nd }}_{\infty} @>\id>> \ell^{2^{j\nd }}_{1}\, .
\end{CD}
 \]
But $\|\id_1: \mmbet\to \ell^{2^{j\nd }}_{\infty} \|=\|\id_2: \ell^{2^{j\nd }}_{1}\to \mmbzt\|=1$, so 
\[ \nn{\id_j: \mmbet\hookrightarrow \mmbzt} \le \nn{\id:\ell_\infty^{2^{j\nd }}\rightarrow \ell_1^{2^{j\nd }}}= 2^{j\nd },
\]
applying \eqref{i1} and \eqref{tong-res} with $n=2^{j\nd }$.


{\em Step 3.}\quad Next we deal with the estimate from above in the last case of \eqref{nu-norm-m-1}.
Here we use the following commutative diagram
\[
\begin{CD}
\mmbet@>\id_j>> \mmbzt\\
@V{\id_1}VV @AA{\id_2}A\\
\ell^{2^{j\nd}}_{1} @>\id>> \ell^{2^{j\nd }}_{\infty}\, .
\end{CD}
 \]
But $\|\id_1: \mmbet\to \ell^{2^{j\nd}}_{1}\|=2^{j\nd (1-\frac{1}{u_1})}$,  $\|\id_2:  \ell^{2^{j\nd }}_{\infty}\to \mmbzt\|=2^\frac{j\nd }{u_2}$  and $ \nn{\id:\ell_1^{2^{j\nd }}\rightarrow \ell_\infty^{2^{j\nd }}}=1$ by \eqref{tong-res}. Consequently, 
\[ \nn{\id_j: \mmbet\hookrightarrow \mmbzt} \le 2^{j\nd (1-\frac{1}{u_1}+\frac{1}{u_2})} \nn{\id:\ell_1^{2^{j\nd }}\rightarrow \ell_\infty^{2^{j\nd }}}= 2^{j\nd (1-\frac{1}{u_1}+\frac{1}{u_2})} .\]

{\em Step 4}.\quad Finally we are left to show the upper estimate in \eqref{nu-norm-m-2}. Here we were inspired by some ideas from the proof of $\nn{\id: \ell_1^n\to\ell^n_\infty}$ in \cite{Pie-op-2}, see also \cite[pp. 58/59]{koenig-78}, and combined it with some Morrey-adapted construction similar to \cite[Substep 2.4 on page 1316]{HaSk-bm1}.   
Let
\begin{equation}\label{def-nu0}
  \nu_0=\nu_0(j)=\min\left\{ \nu\in \N:  2^{\nu \nd }\ge 2^{j\nd \frac{p_2}{u_2}} \right\}.
\end{equation}
Thus $p_2<u_2$ implies $\nu_0\leq j$. Moreover, $2^{j\nd (1-\frac{p_2}{u_2})}\ge 2^{(j-\nu_0)\nd }$. We have $2^{(j-\nu_0)\nd }$ cubes $Q_{-\nu_0,m}$ contained in  $ Q_{-j,0}$. We consider the following family  of sequences 
\begin{align*}
  \mathcal{E} = & \left\{\varepsilon=\{\varepsilon_k\}_{k\in \mathcal{K}_j}: \quad  \varepsilon_k \ \in \{-1,0,1\}, \quad 
  \text{and  for any cube} \right.\\ 
& ~ \left. Q_{-\nu_0,m}\subset Q_{-j,0} \quad \text{there is exactly one} \; k\; \text{such that} \; \varepsilon_k\not= 0 \right\}.
\end{align*}

One can easily verify that $\|\varepsilon|\mmbzt\|=1$  for any $\varepsilon\in \mathcal{E}$: for smaller cubes $Q_{-\nu,m'}\subset Q_{-\nu_0,m}$, that is,  with $0\leq \nu\leq \nu_0$, there is at most one non-vanishing coefficient inside, thus
$$  |Q_{-\nu,m'}|^{\frac{1}{u_2}-\frac{1}{p_2}} \Big(\sum_{k:Q_{0,k}\subset Q_{-\nu,m'}}\!\!|\varepsilon_k|^{p_2}\Big)^{\frac{1}{p_2}} \leq |Q_{-\nu,m'}|^{\frac{1}{u_2}-\frac{1}{p_2}} \leq 1,
  $$  
while for bigger cubes $Q_{-\nu,m'}$, that is, $Q_{-\nu_0,m}\subset Q_{-\nu,m'}\subset Q_{-j,0}$ with $\nu_0\leq \nu\leq j$, there are exactly  $2^{(\nu-\nu_0)\nd }$ non-vanishing coefficients in the corresponding sum, 
\begin {align*}
  |Q_{-\nu,m'}|^{\frac{1}{u_2}-\frac{1}{p_2}} \Big(\sum_{k:Q_{0,k}\subset Q_{-\nu,m'}}\!\!|\varepsilon_k|^{p_2}\Big)^{\frac{1}{p_2}}  &\leq |Q_{-\nu,m'}|^{\frac{1}{u_2}-\frac{1}{p_2}} 2^{(\nu-\nu_0)\frac{\nd }{p_2}}\\
  &= 2^{\nu \nd (\frac{1}{u_2}-\frac{1}{p_2}+\frac{1}{p_2}) - \nu_0 \frac{\nd }{p_2}} 
  \leq 2^{\frac{\nd }{p_2}(\frac{p_2}{u_2}j-\nu_0) } \  \leq 1.
  \end{align*}  
  Here we used \eqref{def-nu0}. Thus \eqref{mseq-defi} leads to
  \begin{equation}\label{norm-eps}
    \|\varepsilon|\mmbzt\|=1.
    \end{equation}

Let us now fix for a moment a certain cube $Q_{-\nu_0,m_0}\subset Q_{-j,0}$. For any $Q_{0,k}\subset Q_{-\nu_0,m_0}$ we shall denote by $\mathcal{E}_k$  the subset of $\mathcal{E}$ that consists of all sequences with $\varepsilon_k\not=0$. Then $\mathcal{E}_{k_1}\cap \mathcal{E}_{k_2} = \emptyset $ if $k_1\not= k_2$ and all the sets $\mathcal{E}_k$ have the same number of elements $|\mathcal{E}_{k}|= \alpha_0$. In consequence   
$$\mathcal{E}=\bigcup_{k: Q_{0,k}\subset Q_{-\nu_0,m_0}}\mathcal{E}_k  $$
and   \begin{equation}\label{meas-E}
	|\mathcal{E}| = \sum_{k: Q_{0,k}\subset Q_{-\nu_0,m_0}} |\mathcal{E}_k| = \alpha_0 \ 2^{\nu_0 \nd }. 
\end{equation}

Let $\lambda\in \mmbet$ and denote by  $(\sum_{\varepsilon\in \mathcal{E}} \varepsilon(\lambda)\varepsilon)_k$ the $k$-th coordinate of the image of $\lambda$ in the linear mapping $\sum_{\varepsilon\in \mathcal{E}} \varepsilon\otimes \varepsilon$. One can easily check that 
\begin{align*}
\Big(\sum_{\varepsilon\in \mathcal{E}} \varepsilon(\lambda)\varepsilon\Big)_k = &  \Big(\sum_{\varepsilon\in \mathcal{E}_k}  \varepsilon(\lambda)\varepsilon\Big)_k = 
|\mathcal{E}_k| \lambda_k + \sum_{\varepsilon\in \mathcal{E}_k} \sum_{i\not=k}\varepsilon_i\varepsilon_k \lambda_i = |\mathcal{E}_k| \lambda_k  = \alpha_0 \lambda_k .
\end{align*}
Hence 
the family $\mathcal{E}$ gives the following representation of $\id_j$
\[ \id_j = \frac{2^{\nu_0 \nd }}{|\mathcal{E}|} \ \sum_{\varepsilon\in \mathcal{E}} \varepsilon\otimes \varepsilon,\] 
where we also applied \eqref{meas-E}.
Thus \eqref{nuclearnorm} leads to 
\begin{equation}\label{nu-est-1}
  \nn{\id_j} \le \frac{2^{\nu_0 \nd }}{|\mathcal{E}|} \ \sum_{\varepsilon\in \mathcal{E}} \left\|\varepsilon|(\mmbet)^* \right\| \left\|\varepsilon|\mmbzt\right\| \le  2^{\nu_0 \nd } \ \sup_{\varepsilon\in \mathcal{E}} \left\|\varepsilon|(\mmbet)^*\right\|,
  \end{equation}
using \eqref{norm-eps}.   It remains to estimate  $\|\varepsilon|(\mmbet)^*\|$ for $\varepsilon \in \mathcal{E}$. 
Let $\|\lambda|\mmbet\|=1$. Then

\begin{align*}
  |\varepsilon(\lambda)| \ & =  \Big|\sum_{k\in \mathcal{K}_j}  \varepsilon_k\lambda_k \Big| 
   \le   \Big(\sum_{k\in \mathcal{K}_j}  |\varepsilon_k|^{p_1'}\Big)^{1/p_1'} 
\Big(\sum_{k\in \mathcal{K}_j}  |\lambda_k|^{p_1}\Big)^{1/p_1} \\
& \leq \Big(\sum_{k\in \mathcal{K}_j}  |\varepsilon_k|^{p_1'}\Big)^{1/p_1'} |Q_{-j,0}|^{\frac{1}{p_1}-\frac{1}{u_1}} |Q_{-j,0}|^{\frac{1}{u_1}-\frac{1}{p_1}} \Big(\sum_{k\in \mathcal{K}_j} |\lambda_k|^{p_1}\Big)^{1/p_1} \\
& \leq   2^{j\nd (\frac{1}{p_1}-\frac{1}{u_1}) } \left\|\lambda |\mmbet\right\| \Big(\sum_{k\in \mathcal{K}_j}  |\varepsilon_k|^{p_1'}\Big)^{1/p_1'}  \\ 
& =   2^{j\nd (\frac{1}{p_1}-\frac{1}{u_1}) }  \Big( |Q_{-j,0}|^{1-\frac{p_2}{u_2}} |Q_{-j,0}|^{\frac{p_2}{u_2}-1} \sum_{k\in \mathcal{K}_j}  |\varepsilon_k|^{p_2}\Big)^{1/p_1'} \\
& \leq   2^{j\nd (\frac{1}{p_1}-\frac{1}{u_1}) } |Q_{-j,0}|^{(1-\frac{p_2}{u_2})(1-\frac{1}{p_1})} 
	\left\|\varepsilon |\mmbzt\right\|^{\frac{p_2}{p_1'}}  \  \\
&=  2^{j\nd (\frac{1}{p_1}-\frac{1}{u_1} + (1-\frac{p_2}{u_2})(1-\frac{1}{p_1}))}  
=   2^{j\nd (1 - \frac{1}{u_1}+\frac{p_2}{u_2p_1})} \ 2^{- j\nd  \frac{p_2}{u_2}}   \\
& \leq   2^{j\nd (1 - \frac{1}{u_1}+\frac{p_2}{u_2p_1})} \ 2^{- (\nu_0-1)\nd }  
\leq  c 2^{-\nu_0\nd } 2^{j\nd (1 - \frac{1}{u_1}+\frac{p_2}{u_2p_1})},
\end{align*}
where we used  $\|\lambda|\mmbet\|=1$, $|\varepsilon_k|\in \{0,1\}$, \eqref{norm-eps}, and finally \eqref{def-nu0} again. 
So  $\big\|\varepsilon|(\mmbet)^*\big\|\le c  2^{-\nu_0\nd } 2^{j\nd (1 - \frac{1}{u_1}+\frac{p_2}{u_2p_1})}$ and \eqref{nu-est-1} results in 
\[\nn{\id_j} \le c\ 2^{j\nd (1 - \frac{1}{u_1}+\frac{p_2}{u_2p_1})} \] 
as desired. This concludes the proof.
\end{proof}

\begin{remark}
In case of $p_i=u_i$, $i=1,2$, Proposition~\ref{nuclear-m} coincides with Tong's result \eqref{tong-res}. Note that the situation \eqref{nu-norm-m-2} cannot appear in that setting, but refers to a `true' Morrey situation for the target space $\mmbzt$. 
  \end{remark}

For further use we reformulate Proposition~\ref{nuclear-m} involving also the number $\tn(r_1,r_2)$ introduced in \eqref{tongnumber}.

\begin{corollary}\label{nuclear-mc}
Let $1\le p_i \le u_i<\infty $ or $p_i=u_i=\infty$, $i=1,2$. Then 
\begin{equation} \label{nuclear-finite}
\nn{\id_j: \mmbet\hookrightarrow \mmbzt}  \sim 
2^{j\nd \big(1-(\frac{1}{u_1} -\frac{1}{u_2}\min\{1,\frac{p_2}{p_1}\})_+\big)} \sim 
2^{\frac{j\nd }{{\tn}(u_1, u_2\max\{1,\frac{p_2}{p_1}\})}}  .
\end{equation}
\end{corollary}

\subsection{The main result for Morrey sequence spaces}

Let $Q$ be a unit cube, $0<p\leq u<\infty$ or $p=u=\infty$, $\sigma\in\R$, $0<q\leq\infty$.  We define a sequence space $\mbt(Q)$ putting 
\begin{multline}
\mbt(Q) :=   \Bigg\{ \lambda = 
\{\lambda_{j,m}\}_{j,m} : \quad     \lambda_{j,m} \in \C ,\quad  j\in\N_0, \  m\in\Zn, \quad Q_{j,m}\subset Q,\quad\text{and}
\\
\|  \lambda \, |\mbt\| = \left(\sum_{j=0}^\infty 2^{jq(\sigma-\frac{\nd }{u}) }  
\!\!\!\sup_{\nu: \nu \le j; k: Q_{\nu,k}\subset Q}\! 2^{q\nd (j-\nu)(\frac 1 u - \frac 1 p )}\Big(\sum_{m:Q_{j,m}\subset Q_{\nu,k}}\!\!|\lambda_{j,m}|^p\Big)^{\frac q p}\right)^{\frac1q} \!\!< \infty \Bigg\},
\label{3-0a}
\end{multline}
with the usual modification when $q=\infty$. 

Moreover, for fixed $j\in \N_0$, we put  
\begin{align*}
\mmbj = \{ \lambda = 
\{\lambda_{j,m}\}_{m\in\Zn} :  \  \lambda_{j,m} \in \C, \quad  Q_{j,m}\subset Q \quad \text{and}\quad\|\lambda|\mmbj\|< \infty\},\\
\intertext{where} 
\|\lambda|\mmbj\| = \sup_{\nu: \nu \le j; k\in \Z^\nd }\! 2^{\nd (j-\nu)(\frac 1 u - \frac 1 p )}\Big(\sum_{m:Q_{j,m}\subset Q_{\nu,k}\subset Q}\!\!|\lambda_{j,m}|^p\Big)^{\frac 1 p}.
\end{align*}

\begin{theorem}  \label{comp-s}
	Let  $\sigma_i\in \R$, $1\leq q_i\leq\infty$, $1\leq p_i\leq u_i< \infty$, or $p_i=u_i=\infty$, $i=1,2$.  Then the embedding 
	\begin{equation} \label{bd1comp}
	\Id: \mbet\hookrightarrow \mbzt
	\end{equation}
	is nuclear if,  and only if,   
	the following condition holds
	\begin{equation}\label{bd3acomp}
 \sigma_1-\sigma_2>  \frac{\nd }{u_1}-\frac{\nd }{u_2} + \frac{\nd }{\tn(u_1,\max(1,\frac{p_1}{p_2})u_2)}.
	\end{equation}
\end{theorem}


\begin{proof}
\emph{Step 1}.~ First we prove the sufficiency of the condition \eqref{bd3acomp} for the nuclearity of $\Id$. We decompose $\Id$ into the sum 
$\Id=\sum_{j=0}^\infty \Id_j$ where the operators $\Id_j: \mbet\hookrightarrow \mbzt$ are defined in the following way 
\begin{equation}
(\Id_j\lambda)_{i,k}= 
\begin{cases} 
\lambda_{j,k} & \text{if}\qquad i=j,\\
0 & \text{if}\qquad i\not=j\, .
\end{cases}
\end{equation}

Now we  factorise the operator $\Id_j$ through  the embedding of finite-dimensional spaces $\id_j:\mmbet\rightarrow \mmbzt $ . Namely we have  the following commutative diagram,
\[
\begin{CD}
\mbet@>\Id_j>> \mbzt\\
@V{P_j}VV @AA{S_j}A\\
\mmbjet@>\id_j >> \mmbjzt\, 
\end{CD}
\]
where $P_j$ denotes the projection of $\mbet$ onto the $j$-level, and $S_j$ is the natural injection of the finite-dimensional $j$-level spaces into $\mbzt$.  
It is easy to check that  $\|P_j: \mbet\to \mmbjet\|= 2^{-j (\sigma_1- \frac{\nd }{u_1})}$ and $\|S_j:\mmbjzt\to \mbzt\|= 2^{j (\sigma_2- \frac{\nd }{u_2})}$.
We use the notation $\delta=\sigma_1-\sigma_2-\frac{\nd }{u_1}+\frac{\nd }{u_2}$. Now  \eqref{i1} and \eqref{nuclear-finite} lead to 
\[ \nn{\Id_j: \mbet\hookrightarrow \mbzt} \le  2^{j (\sigma_2- \sigma_1+ \frac{\nd }{u_1}- \frac{\nd }{u_2})}  2^{j\nd \big(1-(\frac{1}{u_1} -\frac{1}{u_2}\min\{1,\frac{p_2}{p_1}\})_+\big)}  =   2^{-j\delta+j\nd \big(1-(\frac{1}{u_1} -\frac{1}{u_2}\min\{1,\frac{p_2}{p_1}\})_+\big)} .\]
Thus
\begin{align} 
\nn{\Id}\le \sum_{j=0}^\infty \nn{\Id_j} \le  \sum_{j=0}^\infty 2^{-j\nd  \left( 		\frac{\delta}{\nd }-\big(1-(\frac{1}{u_1} -\frac{1}{u_2}\min\{1,\frac{p_2}{p_1}\})_+\big)\right)} < \infty 
\end{align}
since  $\frac{\delta}{\nd } >  1-(\frac{1}{u_1}-\frac{1}{u_2}\min\{1,\frac{p_2}{p_1}\})_+
=  \frac{1}{\tn(u_1,\max\{1,\frac{p_1}{p_2}\}u_2)}$ by \eqref{bd3acomp} and \eqref{tongnumber}.\\ 

\emph{Step 2}. 
 We come to the necessity of \eqref{bd3acomp} and assume that the  embedding \eqref{bd1comp} is nuclear. We consider the following diagram 
\begin{equation}\label{0803_0}
\begin{CD}
\mmbjet@>\id_j >> \mmbjzt\\
@V{\widetilde{S}_j}VV @AA{\widetilde{P}_j}A\\
\mbet@>\Id>> \mbzt\, ,
\end{CD}
\end{equation}
where  the operators $\widetilde{P}_j$ and $\widetilde{S}_j$ are defined in a similar way to $P_j$ and $S_j$,  but now for the spaces $\mmbjzt$ and $\mbzt$,  or $\mmbjet$ and $\mbet$, respectively. We have   $\|\widetilde{P}_j:\mbzt\to \mmbjzt\|= 2^{-j (\sigma_2- \frac{\nd }{u_2})}$ and $\|\widetilde{S}_j: \mmbjet\to \mbet\|= 2^{j (\sigma_1- \frac{\nd }{u_1})}$, such that by Corollary~\ref{nuclear-mc}, 
\begin{equation} \label{0803_00}
  2^{\frac{j\nd }{{\tn}(u_1, u_2\max\{1,\frac{p_2}{p_1}\})}}
  \le C \nn{\id_j}\le C \left\|\widetilde{P}_j:\mbzt\to \mmbjzt\right\| \left\|\widetilde{S}_j: \mmbjet\to \mbet\right\| \nn{\Id} \le C 2^{j\delta}, \quad j\in\no. 
\end{equation} 
This immediately implies  
\[ 
  \frac{\delta}{\nd } \ge    \frac{1}{\tn(u_1,\max\{1,\frac{p_1}{p_2}\}u_2)}
\]
and we are left to disprove equality of the terms above to complete the proof of \eqref{bd3acomp}. \\

\emph{Step 3. }
So assume to the contrary  that  $\tn(u_1,\max\{1,\frac{p_1}{p_2}\}u_2) = \nd/\delta$ 
and the  operator $\Id$ is nuclear. To simplify the notation we can put $\sigma_1=\frac{\nd}{u_1}$. This follows easily from the following commutative diagrams
\begin{align*}
	\begin{CD}
		\mbet@>\Id>> \mbzt\\
		@V{D_1}VV @AA{D_2}A\\
		\widetilde{n}^{\nd/u_1}_{u_1,p_1,q_1}@>\id >>\widetilde{n}^{\sigma_3}_{u_2,p_2,q_2}\, 
	\end{CD}
\quad \qquad\text{and}\quad\qquad
\begin{CD}
	\widetilde{n}^{\nd/u_1}_{u_1,p_1,q_1}@>\id >>\widetilde{n}^{\sigma_3}_{u_2,p_2,q_2}\\
	@V{D_2}VV @AA{D_1}A\\
	\mbet@>\Id>> \mbzt\, 
\end{CD}
\end{align*} 
where $D_1: (\lambda_{j,m})\mapsto (2^{j(\sigma_1-\frac{\nd}{u_1})}\lambda_{j,m})$, $D_2: (\lambda_{j,m})\mapsto (2^{-j(\sigma_1-\frac{\nd}{u_1})}\lambda_{j,m})$ and $\sigma_3= \sigma_2-\sigma_1+\frac{\nd}{u_1}$. So we assume $\sigma_1=\frac{\nd}{u_1}$ in the sequel. In particular, this implies that $\|P_j: \mbet\to \mmbjet\|= 2^{-j (\sigma_1- \frac{\nd }{u_1})}=1$ and $\|\widetilde{S}_j: \mmbjet\to \mbet\|= 2^{j (\sigma_1- \frac{\nd }{u_1})}=1$.

Then, by definition, there exist functionals $f^{(k)}\in \big(\mbet \big)'$ and  vectors $e^{(k)}\in \mbzt$, $\|e^{(k)}|\mbzt\|=1$, such that 
\begin{align}\label{0803_1}
	\Id(\lambda)=\sum_{k=1}^\infty f^{(k)}(\lambda)e^{(k)} \qquad \text{and} \qquad \sum_{k=1}^\infty \left\|f^{(k)}| \big(\mbet \big)' \right\| =C<\infty . 
\end{align}

If $q_1<\infty$, then  it follows from the properties of vector-valued $\ell_q$ spaces that the dual space $\big(\mbet\big)'$ coincides with the  space 
\[\ell_{r_1}\Big( (\mmbjet)'\Big),\qquad \frac{1}{r_1}=1-\frac{1}{q_1},  \] 
cf. \cite[Lemma 1.11.1]{Tri-int}. In consequence any functional  $f^{(k)}$  can be represented by a sequence $\left\{f^{(k)}_j\right\}_{j=1}^\infty$, $f^{(k)}_j\in (\mmbjet)'$, and    
\begin{equation}\label{0803_2}
f^{(k)}(\lambda)= \sum_{j=0}^\infty f^{(k)}_j(P_j(\lambda)) .
\end{equation} 
Moreover,
\begin{equation} \label{0803_4}
	\left\|f^{(k)}| \big(\mbet \big)' \right\| \sim \Big(\sum_{j=0}^\infty  \|f^{(k)}_j|(\mmbjet)'\|^{r_1}\Big)^\frac{1}{r_1}, 
      \end{equation}
      with the usual modification when $r_1=\infty$.  

Let us choose $\varepsilon>0$. It follows from \eqref{0803_1} and \eqref{0803_4} that there is a number $M_0=M_0(\varepsilon)\in \N$ such that for any $j\in \N_0$ the following inequality holds 
\begin{equation}\label{15_09_22}
	  \sum_{k=M_0}^\infty \|f^{(k)}_j|(\mmbjet)'\|\le  \sum_{k=M_0}^\infty \left\|f^{(k)}| \big(\mbet \big)' \right\|\le \frac{\varepsilon}{2}\ .
\end{equation}  
Moreover, \eqref{0803_2} and \eqref{0803_4} imply that we can choose $N_0=N_0(M_0(\varepsilon))\in\N$ depending on $M_0$ such that for any $k=1,\ldots, M_0-1$ we have 
\begin{align}
	\Big\|f^{(k)}- \sum_{j=0}^{N_0} f^{(k)}_j\circ P_j \Big| \big(\mbet \big)'\Big\| = & \  
	\Big\|\sum_{j=N_0}^\infty f^{(k)}_j\circ P_j \Big| \big(\mbet \big)'\Big\| \nonumber\\ \le & \ \Big(\sum_{j=N_0}^\infty  \|f^{(k)}_j|(\mmbjet)'\|^{r_1}\Big)^\frac{1}{r_1}\le  \frac{\varepsilon}{2^{k+1}}.
  \label{15-09-22-2}
  \end{align}

We take $j_0> N_0$ and $\lambda\in m^{(j_0)}_{u_1,p_1}$.  Then $f^{(k)}(\widetilde{S}_{j_0}(\lambda))= f^{(k)}_{j_0}(\lambda)$ and $f_j^{(k)}(P_j\circ \widetilde{S}_{j_0}(\lambda))= 0$ if $j\not= j_0$.  So the identity in \eqref{0803_1} implies that
\begin{equation}\label{15_09_22_3}
	 \widetilde{S}_{j_0}(\lambda) =\sum_{k=1}^\infty f^{(k)}_{j_0}(\lambda) e^{(k)} .
\end{equation}
Hence, using \eqref{0803_2}, we get 
\begin{align*}
  |f^{(k)}_{j_0}(\lambda)| =	 |f^{(k)}(\widetilde{S}_{j_0}(\lambda))|= & \ |(f^{(k)}(\widetilde{S}_{j_0}(\lambda))- \sum_{j=0}^{N_0} f^{(k)}_j)(P_j\circ \widetilde{S}_{j_0}(\lambda))| \nonumber \\
  \le  & \ \Big\|f^{(k)}- \sum_{j=0}^{N_0} f^{(k)}_j\circ P_j \Big| \big(\mbet \big)'\Big\|\, \|\widetilde{S}_j: \mmbjet\to \mbet\| \|\lambda|\mmbjet\| \nonumber\\
\le & \ \frac{\varepsilon}{2^{k+1}}  \|\lambda|\mmbjet\| 
\end{align*}
for any $k=1,\ldots, M_0-1$, recall \eqref{15-09-22-2} and $\|\widetilde{S}_j: \mmbjet\to \mbet\|=1$. Thus 
\begin{eqnarray}\label{15_09_22_4}
	\sum_{k=1}^\infty \|f^{(k)}_{j_0}|(\mmbjet)'\|\le \varepsilon,
\end{eqnarray} 
 if $j_0> N_0$. 
 

Now we can benefit once more from the commutative diagram \eqref{0803_0}, and using \eqref{15_09_22_3}-\eqref{15_09_22_4} we get 
\begin{align}\label{25_10_22_1}
\id_{j_0}(\lambda) =\widetilde{P}_{j_0}\circ \Id\circ\widetilde{S}_{j_0}(\lambda) =
 \sum_{k=1}^\infty f^{(k)}_{j_0}(\lambda) \widetilde{P}_{j_0}(e^{(k)}) 
 \intertext{and} 
  \sum_{k=1}^\infty \left\|f^{(k)}_{j_0}| \big(\mmbjet\big)' \right\|\,  \left\|\widetilde{P}_{j_0}(e^{(k)})| \mmbjzt \right\| \le \varepsilon  2^{-j_0(\sigma_2-\frac{\nd}{u_2})}, \label{25_10_22_2}
\end{align}
in view of $\|\widetilde{P}_{j_0}:\mbzt\to \mmbjzt\|= 2^{-j_0 (\sigma_2- \frac{\nd }{u_2})}$, $ \|e^{(k)}|\mbzt\|=1$, and \eqref{15_09_22_4}. In other words, \eqref{25_10_22_1}, \eqref{25_10_22_2} represent a nuclear representation of $\id_{j_0}$ such that $\nn{\id_{j_0}} \le C 2^{-j_0(\sigma_2-\frac{\nd}{u_2})} $. Finally, arguing in the same way as in \eqref{0803_00}, we get   
\[ 2^{j_0\delta}\le C\nn{\id_{j_0}} \le C 2^{-j_0(\sigma_2-\frac{\nd}{u_2})} \varepsilon\, . \]
But  $\delta= \frac{\nd}{u_2}- \sigma_2$ for $\sigma_1=\frac{\nd}{u_1}$, so we get a contradiction. This  proves the necessity  for $q_1<\infty$. 

If $q_1=\infty$,  then the statement follows from simple inclusions
$
\widetilde{n}^{\sigma_1}_{u_1,p_1,q} \hookrightarrow {\widetilde{n}^{\sigma_1}_{u_1,p_1,\infty}}\hookrightarrow  {\widetilde{n}^{\sigma_2}_{u_2,p_2,q_2}}$, $\ q<\infty $.
\end{proof}

  \begin{remark}
    Note that in case of $p=u$ the spaces $\widetilde{n}^{\sigma}_{u,p,q}$ coincide with the corresponding Besov sequence spaces $\widetilde{b}^\sigma_{p,q}$. For such (and more general) vector-valued sequence spaces we obtained in \cite{HaLeoSk} a characterisation for their embeddings to be nuclear. In our setting, that is, Theorem~\ref{comp-s} with $p_i=u_i$, $i=1,2$, the result then reads as
    \[
\Id: \widetilde{n}^{\sigma_1}_{p_1,p_1,q_1} \hookrightarrow \widetilde{n}^{\sigma_2}_{p_2,p_2,q_2} \quad\text{nuclear}\quad \iff\ \left\{2^{-j(\delta-\frac{\nd}{\tn(p_1,p_2)})}\right\}_{j\in\nat} \in \ell_{\tn(q_1,q_2)},
    \]
    where for $\tn(q_1,q_2)=\infty$ the space $\ell_\infty$ has to be replaced by $c_0$. Here $\delta = \sigma_1-\sigma_2-\frac{\nd}{p_1}+\frac{\nd}{p_2}$ and $\tn(r_1,r_2)$ is given by \eqref{tongnumber}. Hence, in this special case,  $\Id$ given by \eqref{bd1comp} is nuclear if, and only if,
    \[
    \delta> \frac{\nd}{\tn(p_1,p_2)} \iff \frac{\sigma_1-\sigma_2}{\nd} > 1- \left(\frac{1}{p_2}-\frac{1}{p_1}\right)_+
    \]
and thus \eqref{bd3acomp} represents the sufficient and necessary condition for the nuclearity of $\Id$.   
  \end{remark}

  \section{Nuclear embeddings of Morrey smoothness spaces}\label{nuc-morrey-func}

  We finally combine our findings from the preceding sections to obtain the desired counterparts of the compactness results in Theorems~\ref{comp} and \ref{comp-tau}.

\begin{theorem}  \label{comp-NE}
	Let  $s_i\in \R$, $1\leq q_i\leq\infty$, $1\leq p_i\leq u_i<\infty$, or $p_i=u_i=\infty$, $i=1,2$.  Then the embedding 
	\begin{equation} \label{bd1comp-N}
        \id_{\mathcal A}: \MAe(\Omega)\hookrightarrow \MAz(\Omega)
	\end{equation}
	is nuclear if, 
	 and only if,	the following condition holds
	 \begin{equation}\label{bd3acomp-suff}
\frac{s_1-s_2}{\nd} > 
\frac{1}{u_1}-\frac{1}{u_2}+\frac{1}{\tn(u_1,\max\{1,\frac{p_1}{p_2}\}u_2)}.
	 \end{equation}
\end{theorem}

\begin{proof}
  First observe that in view of the independence of \eqref{bd3acomp-suff}  	from $q_i$, $i=1,2$, together with the elementary embeddings \eqref{elem} and property \eqref{i1}, it is sufficient to prove the result for $\MA=\MB$. Moreover, Theorem~\ref{comp-s} and the wavelet decomposition in Theorem~\ref{wavemorrey} imply both the sufficiency and  the necessity of the   assumption \eqref{bd3acomp-suff} for the function spaces.
\end{proof}

  \begin{remark}\label{R-compact-nuc-MA}
    We compare our compactness result in Theorem~\ref{comp} with the nuclearity one in Theorem~\ref{comp-NE}. It turns out that compactness and nuclearity criteria for $\id_{\mathcal A}$ coincide if
    \[
      \Big(\frac{1}{u_2}-\frac{1}{\max\{1,p_2/p_1\}u_1}\Big)_+ =
      \begin{cases} 1, & \text{if}\ \max\{1,\frac{p_1}{p_2}\}u_2 \leq u_1,\\ 1-\frac{1}{u_1}+\frac{1}{\max\{1,\frac{p_1}{p_2}\}u_2},& \text{if}\ u_1\leq \max\{1,\frac{p_1}{p_2}\}u_2,\end{cases}
      \]
using \eqref{tongnumber}. 
  This is equivalent to the two cases
  $1=u_1=p_1\leq p_2 \leq u_2=\infty$, or 
 $1=u_2=p_2 \leq p_1 \leq u_1=\infty$,
where we always assume that either $1\leq p_i\leq u_i<\infty$ or $p_i=u_i=\infty$, $i=1,2$. So we are led to the natural extension of the parallel observation for the embedding $\id_\Omega$ in \eqref{id_Omega-nuclear-0}, see the end of Remark~\ref{prod-id_Omega-nuc}. Moreover, since $u_i=\infty$ for $i=1$ or $i=2$ is required, this applies only to the situation of $\MA=\MB$.   \end{remark}

  Next we benefit from Proposition~\ref{yy02}, \eqref{010319}, \eqref{N-BT-emb}, \eqref{N-BT-equal}, \eqref{fte}, together with \eqref{elem-tau} to obtain a nuclearity result for spaces $\at(\Omega)$. We 
adapt, for convenience, our above notation \eqref{gamma} and introduce the following abbreviation:
\begin{align}
\critical 
  &
    :=  
    1- \max\left\{\left(\tau_1-\frac{1}{p_1}\right)_+ -\left(\tau_2-\frac{1}{p_2}\right)_+, \frac{1}{p_2} -\tau_2 - \min\left\{\frac{1}{p_1}-\tau_1, \frac{1}{p_1}(1-p_2\tau_2)_+\right\}     \right\}\nonumber\\
\label{gamma_nuc}
& =  \begin{cases}
1+\frac{1}{p_1}-\tau_1-\frac{1}{p_2}+\tau_2, & \text{if}\quad \tau_1\geq \frac{1}{p_1}, \\[1ex]
1-\frac{1}{p_2}+\tau_2, &\text{if}\quad  \tau_1< \frac{1}{p_1}, \ \tau_2\geq  \frac{1}{p_2}, \\[1ex]
1-
\max\left\{0,\frac{1}{p_2}-\tau_2-\frac{1}{p_1}+\max\{\tau_1,\frac{p_2}{p_1}\tau_2\}\right\}, &\text{if}\quad  \tau_1< \frac{1}{p_1}, \  \tau_2< \frac{1}{p_2}.
\end{cases}
\end{align}
Here and in the sequel we put $p_i\tau_i=1$ in case of $p_i=\infty$ and $\tau_i=0$. Similarly we shall understand  $\frac{p_i}{p_k}=1$ if $p_i=p_k=\infty$. Note that $p_i\geq 1$ and $\tau_i\geq 0$ immediately imply that $\critical$ given by \eqref{gamma_nuc} exceeds the number $\gamma(\tau_1,\tau_2,p_1,p_2)$ given by \eqref{gamma} (which characterised the compactness condition) as it should be, since nuclearity is stronger than compactness.\\

Then the counterpart of Theorem~\ref{comp-tau} is the following.

\begin{theorem}\label{C-nuc-at}
Let $s_i\in \real$, $1\leq q_i\leq\infty$, $1\leq p_i\leq \infty$ (with $p_i<\infty$ in case of $A=F$), $\tau_i\geq 0$, $i=1,2$. Then 
\[
\id_\tau : \ate(\Omega) \hookrightarrow \atz(\Omega)
\]
is nuclear if, and only if,  
\begin{equation}\label{tau-nuc}
  \frac{s_1-s_2}{\nd} > \critical.
  \end{equation}
\end{theorem}

\begin{proof} We consider several cases. \\
	\emph{Case 1}.~ If $\tau_i> \frac{1}{p_i}$ or $\tau_i = \frac{1}{p_i}$  and $q_i=\infty$, $i=1,2$, then by Proposition~\ref{yy02}, 
	\[B^{s_i,\tau_i}_{p_i,q_i}(\Omega) = B^{s_i+\nd(\tau_i-\frac{1}{p_i})}_{\infty,\infty}(\Omega).\]
Hence \eqref{id_Omega-nuclear} implies that $\id_\tau$ is nuclear if, and only if, $s_1-s_2> \nd(1 +\frac{1}{p_1}-\tau_1-\frac{1}{p_2}+\tau_2)$ which corresponds to \eqref{tau-nuc} in view of \eqref{gamma_nuc} (first line) in this case. \\
    
\emph{Case 2}.~ Let $0\le \tau_i < \frac{1}{p_i}$,   $i=1,2$, that is, in particular, $p_i<\infty$, $i=1,2$. Recall that $\fti(\Omega)=\MFi(\Omega)$ with $\frac{1}{u_i}= \frac{1}{p_i}-\tau_i$, $i=1,2$, cf. \eqref{fte}.
Due to \eqref{elem-tau} and the independence of \eqref{bd3acomp-suff} from $q_i$, $i=1,2$, Theorem \ref{comp-NE} implies (after some calculations) that $\id_\tau$ is nuclear if, and only if, \eqref{tau-nuc} is satisfied (now referring to the third line of \eqref{gamma_nuc}).\\
  
\emph{Case 3}.~ Next we consider the mixed cases, that is, when one of the spaces satisfies the assumptions of  Case~1 and the other one of Case~2. Assume first that $\tau_2>\frac{1}{p_2}$ or $\tau_2=\frac{1}{p_2}$ with $q_2=\infty$, and $\tau_1<\frac{1}{p_1}$.  Then by Proposition~\ref{yy02}, $\atz(\Omega)=B^{s_2+\nd(\tau_2-\frac{1}{p_2})}_{\infty,\infty}(\Omega)$, and $\id_\tau$ reads as
  \[ \id_\tau: \ate(\Omega)\hookrightarrow B^{s_2+\nd(\tau_2-\frac{1}{p_2})}_{\infty,\infty}(\Omega).   \] 
For the sufficiency in this case we can apply \eqref{010319} and Theorem~\ref{prod-id_Omega-nuc} parallel to Case~1. For the necessity we use \eqref{elem-tau}, \eqref{N-BT-emb} together with Theorem~\ref{comp-NE} and the observation that $B^{s_2+\nd(\tau_2-\frac{1}{p_2})}_{\infty,\infty}(\Omega) = \mathcal{N}^{s_2+\nd(\tau_2-\frac{1}{p_2})}_{\infty,\infty,\infty}(\Omega)$.
Conversely, if $\tau_1>\frac{1}{p_1}$ or $\tau_1=\frac{1}{p_1}$ with $q_1=\infty$, and $\tau_2<\frac{1}{p_2}$, then $\id_\tau$ becomes
\[ \id_\tau: \ate(\Omega)= B^{s_1+\nd(\tau_1-\frac{1}{p_1})}_{\infty,\infty}(\Omega)   \hookrightarrow \atz(\Omega).\] 
Now the necessity immediately follows by \eqref{010319} and Theorem~\ref{prod-id_Omega-nuc}, whereas the sufficiency is a consequence of \eqref{N-BT-emb} (in the $B$-case) or \eqref{fte} (in the $F$-case) and Theorem~\ref{comp-NE} again.\\

\emph{Case 4}.~ It remains to deal with the limiting situations when $\tau_1=\frac{1}{p_1}$ and $q_1<\infty$ or $\tau_2=\frac{1}{p_2}$ and $q_2<\infty$. Assume first that $\atz(\Omega)$ satisfies the limiting conditions, while $\ate(\Omega)$ does not, i.e., $\tau_2=\frac{1}{p_2}$ with $q_2<\infty$, and $\tau_1<\frac{1}{p_1}$ or $\tau_1\geq \frac{1}{p_1}$ with $q_1=\infty$ if $\tau_1=\frac{1}{p_1}$. For the sufficiency, let \eqref{tau-nuc} be satisfied and choose some  $\tilde{s}_2>s_2$ such that 
  \[s_1-s_2>s_1-\tilde{s}_2> \nd \critical. \] 
  Then 
    \[\ate(\Omega)\hookrightarrow A^{\tilde{s}_2,\tau_2}_{p_2,\infty}(\Omega) \hookrightarrow \atz(\Omega),\]
and the nuclearity follows by \eqref{elem-0-t}, \eqref{i1} and Case~1 or Case~3.  
  Conversely, let $\id_\tau$ be nuclear, then we have 
   \[ \ate(\Omega)\hookrightarrow \atz(\Omega)\hookrightarrow A^{s_2,\tau_2}_{p_2,\infty}(\Omega)\]
by \eqref{elem-1-t}, and  the necessity of the condition  \eqref{tau-nuc} follows once more from \eqref{i1} together with Case~1 or Case~3. 
   
   Finally, let   $\tau_1=\frac{1}{p_1}$ and $q_1<\infty$. As for the sufficiency we proceed similar as above, that is, \eqref{elem-1-t} and the preceding cases imply that
     \[ \ate(\Omega)\hookrightarrow A^{s_1,\tau_1}_{p_1,\infty}(\Omega)\hookrightarrow \atz(\Omega)\]
is nuclear when \eqref{tau-nuc} is satisfied.  

  We come to the necessity and strengthen similar arguments as in the proof of \cite[Theorem~3.2]{ghs20} concerning compactness. If $0\leq \tau_2< \frac{1}{p_2}$ and $\frac{1}{u_2}=\frac{1}{p_2}-\tau_2$, then \cite[Corollary~5.2]{YHSY} (for $A=B$) and \eqref{ftbt} (for $A=F$) lead to
\begin{equation}
B^{s_1}_{\infty,q_1}(\Omega)\hookrightarrow \ate(\Omega) \hookrightarrow \atz(\Omega) \hookrightarrow A^{s_2,\tau_2}_{p_2,\infty}(\Omega) = \mathcal{A}^{s_2}_{u_2,p_2,\infty}(\Omega)\hookrightarrow \mathcal{N}^{s_2}_{u_2,p_2,\infty}(\Omega),
\end{equation}
where we used \eqref{N-BT-equal} and \eqref{fte} in the last equality and \eqref{elem} in the last embedding. In view of Theorem~\ref{comp-NE} this leads to $s_1-s_2>\nd(1-\frac{1}{p_2}+\tau_2)$ as required. 

In the double-limiting case, that is, when also $\tau_2=\frac{1}{p_2}$ and $q_2<\infty$, and $\id_\tau$ is nuclear, choose $q_0\leq \min\{p_1,q_1\}$. Then
\[
B^{s_1}_{\infty,q_0}(\Omega)\hookrightarrow \ate(\Omega) \hookrightarrow \atz(\Omega) \hookrightarrow B^{s_2}_{\infty,\infty}(\Omega)
\]
is nuclear, where we used for the first embedding \cite[Corollary~5.2]{YHSY} (with \eqref{elem-tau} for $A=F$) again, and \eqref{010319} for the last one. But this implies $s_1-s_2>\nd$ in view of Theorem~\ref{prod-id_Omega-nuc}, which is just \eqref{tau-nuc} in this case.

Assume finally $\tau_2\geq \frac{1}{p_2}$  with $q_2=\infty$ if $\tau_2=\frac{1}{p_2}$. Then $\atz(\Omega)=B^{s_2+\nd(\tau_2-\frac{1}{p_2})}_{\infty,\infty}(\Omega)$. Choose $q_0\leq \min\{p_1,q_1\}$. Then due to the assumed nuclearity of $\id_\tau$ the embedding
 \[
   B^{s_1}_{\infty,q_0}(\Omega)\hookrightarrow \ate(\Omega) \hookrightarrow\atz(\Omega )=  B^{s_2+\nd(\tau_2-\frac{1}{p_2})}_{\infty,\infty}(\Omega),
 \]
is also nuclear, where the first embedding is continuous due to \cite[Proposition~2.4]{ysy} (restricted to spaces on domains). So Theorem~\ref{prod-id_Omega-nuc} leads to $s_1-s_2>\nd(1-
\tau_2+\frac{1}{p_2})$ as desired in this case. This completes the proof.
  \end{proof}



  \begin{remark}
    Parallel to Remark~\ref{R-compact-nuc-MA} we compare the compactness result for $\id_\tau$ as recalled in Theorem~\ref{comp-tau} with its nuclearity counterpart, Theorem~\ref{C-nuc-at}. Thus the necessary and sufficient compactness condition \eqref{tau-comp-u2} for $\id_\tau$ coincides with the corresponding one for its nuclearity \eqref{tau-nuc}, if     $\gamma(\tau_1,\tau_2,p_1,p_2)$ given by \eqref{gamma} coincides with $\critical$ given by \eqref{gamma_nuc}, $ \gamma(\tau_1,\tau_2,p_1,p_2)= \critical$.

    Straightforward calculation yields that this is possible if, and only if,
  $\tau_1\geq \frac{1}{p_1}$ and $\tau_2=0$, $p_2=1$, or $\tau_2\geq \frac{1}{p_2}$ and $\tau_1=0$, $p_1=1$, so we are essentially in the classical situation, recall Proposition~\ref{yy02} and the end of Remark~\ref{R-nuc-dom-class}. However, we do not really need the coincidence of the spaces in Proposition~\ref{yy02}: due to the independence of both conditions 
\eqref{tau-comp-u2} and \eqref{tau-nuc} of the fine parameters $q_i$, the additional assumption that $q_i=\infty$ when $\tau_i=\frac{1}{p_i}$ for $i=1$ or $i=2$, is not necessary. But the gain in this Morrey setting is small. In contrast to Remark~\ref{R-compact-nuc-MA} we can also observe this extremal case now for all spaces $\at(\Omega)$ (and not only for the Besov scale $\bt(\Omega)$), as the additional parameter $\tau$ compensates for the otherwise necessary condition that $u_i=\infty$ (or $p_i=\infty$ in the classical situation).
\end{remark}

  \begin{remark}
    We briefly return to the new spaces $\rhoA$ introduced in \cite{HT6} and check whether the so-called {\em Slope-$\nd$-rule} observed in Remark~\ref{comp-rho} for the compactness assertion remains true also for the nuclearity. 
    Then Theorems~\ref{comp-NE} and \ref{C-nuc-at} amount to the following observation. Let $-\nd<\vr<0$, $1\leq p_i<\infty$, $1\leq q_i\leq\infty$, $s_i\in\real$, $i=1,2$. Then the embedding
\[
      \id_{\Omega,\vr} : \rhoAe(\Omega) \hookrightarrow \rhoAz(\Omega) 
  \]
    is nuclear if, and only if,
    \begin{equation}
    	\nonumber
      s_1-s_2 > |\vr| - |\vr| \left(\frac{1}{p_2}-\frac{1}{p_1}\right)_+. 
    \end{equation}
   This is a consequence of Theorems~\ref{comp-NE} and \ref{C-nuc-at} together with the identities mentioned in Remark~\ref{rem-rho-A}. In other words, it corresponds to Theorem~\ref{prod-id_Omega-nuc} for the classical situation when $\vr=-\nd$. Thus it represents another example for the {\em Slope-$\nd$-rule}. 
\end{remark}

  Finally we collect some special cases and begin with the situation when the target space is $\bmo(\Omega)$, recall Remark~\ref{bmo-def}. We can now give the counterpart of the compactness result in Remark~\ref{comp-bmo}.
  
\begin{corollary}\label{nuc-at-bmo}
Let $s\in \real$, $1\leq q\leq\infty$, $1\leq p\leq \infty$ (with $p<\infty$ in case of $A=F$). 
\bli
  \item[{\upshape\bfseries (i)}]
Assume $1\le p\le u < \infty $ or $p=u=\infty$. Then 
\[
\id_\tau : \MA(\Omega) \hookrightarrow \bmo(\Omega)
\]
is nuclear if, and only if, $s>\nd$. 
  \item[{\upshape\bfseries (ii)}]
Assume $\tau\geq 0$. Then 
\[
\id_\tau : \at(\Omega) \hookrightarrow \bmo(\Omega)
\]
is nuclear if, and only if, $s>\nd - \nd\left(\tau-\frac1p\right)_+$. 
\eli
\end{corollary}

\begin{proof}
  The result follows immediately from Theorem~\ref{C-nuc-at} together with \eqref{ft=bmo} (for spaces restricted to $\Omega$).
\end{proof}

\begin{remark}
The result remains true when $\bmo(\Omega)$ is replaced by $L_\infty(\Omega)$ or $C(\Omega)$.
\end{remark}

Another case of some special interest is the case $\tau_1=\tau_2$ in Theorem~\ref{C-nuc-at}. Note that in this case \eqref{gamma_nuc} reads as
\[
  \overline{\gamma}(\tau,\tau,p_1,p_2) = \begin{cases} 1-\frac{1}{p_2}+\frac{1}{p_1}, & \text{if}\ p_1\geq p_2, \\ 
1-\min\left\{0, \frac{1}{p_2}-\min\left\{\tau,\frac{1}{p_1}\right\}\right\}, & \text{if}\ p_1\leq p_2.
  \end{cases}
\]

Now Theorem~\ref{C-nuc-at} implies the following result.

\begin{corollary}\label{C-nuc-tau-tau}
  Let $s_i\in \real$, $1\leq q_i\leq\infty$, $1\leq p_i\leq \infty$ (with $p_i<\infty$ in case of $A=F$), $i=1,2$, and $\tau\geq 0$. 
  Then 
\[
\id_\tau : A^{s_1,\tau}_{p_1,q_1}(\Omega) \hookrightarrow A^{s_2,\tau}_{p_2,q_2}(\Omega)
\]
is nuclear if, and only if,  
\begin{equation}\label{tau-nuc-3}
  \frac{s_1-s_2}{\nd} > \begin{cases} 1-\frac{1}{p_2}+\frac{1}{p_1}, & \text{if}\ p_1\geq p_2, \\[1ex] 
1-\min\left\{0, \frac{1}{p_2}-\min\left\{\tau,\frac{1}{p_1}\right\}\right\}, & \text{if}\ p_1\leq p_2.
  \end{cases}
  \end{equation}
\end{corollary}

\begin{remark}
  When $\tau=0$, Corollary~\ref{C-nuc-tau-tau} corresponds to Theorem~\ref{prod-id_Omega-nuc} and \eqref{tau-nuc-3} coincides with \eqref{id_Omega-nuclear}. So it can be seen as some $\tau$-lifted version of the classical case.
\end{remark}


We conclude the paper with the nuclearity results for the embeddings into spaces $L_r(\Omega)$, $1\leq r<\infty$, which can also be seen as counterpart  of Corollary~\ref{nuc-at-bmo} which refers to $r=\infty$. Furthermore, it corresponds to \cite[Corollary~3.10]{ghs20} where we dealt with the compactness in that situation.

\begin{corollary}\label{nuc-atauinLr}
 Let $s\in\real$, $1\leq p,q\leq\infty$ (with $p<\infty$ if $A=F$), $1\leq r<\infty$.
  \bli
  \item[{\upshape\bfseries (i)\hfill}] Let $\tau\geq 0$. Then
\[
\at(\Omega)\hookrightarrow L_r(\Omega)\]
\text{is nuclear if, and only if,}
\begin{equation}\label{lim-12}
  {  s>\nd \begin{cases} 1 + \frac1p-\tau-\frac1r & \text{if}\quad \tau\geq \frac1p, \\[1ex]
  1-\left(\frac1r-\frac1p+\tau\right)_+ & \text{if}\quad \tau\leq \frac1p.\end{cases}}
\end{equation}
  \item[{\upshape\bfseries (ii)\hfill}] Let $u\in [p, \infty)$ (or $p=u=\infty$ if $\mathcal{A}=\mathcal{N}$). Then
\begin{equation}\label{lim-13}
  \MA(\Omega)\hookrightarrow L_r(\Omega) \quad\text{is nuclear\quad if, and only if,}\quad
\frac{s}{\nd}>1-\left(\frac1r-\frac1u\right)_+\ .
\end{equation}
 \eli
\end{corollary}

\begin{proof}
We first show (i) and start with the case $A=B$. Note that $B^0_{r,1}(\Omega)\hookrightarrow L_r(\Omega)\hookrightarrow B^0_{r,\infty }(\Omega)$, $1\leq r\leq\infty$, so we apply  Theorem~\ref{C-nuc-at} for $s_1=s$, $s_2=0$, $p_1=p$, $p_2=r$, $\tau_1=\tau$, $\tau_2=0$, $q_1=q$, and $q_2=1$ or $q_2=\infty$ to obtain the necessary and sufficient conditions. Again we benefit from the independence of \eqref{tau-nuc-3} with respect to the $q$-parameters. The case $A=F$ follows by \eqref{elem-tau} again.\\
We deal with (ii), which is well-known when $u=p$, recall Theorem~\ref{prod-id_Omega-nuc}. Assume now $u>p$ and let $\tau=\frac1p-\frac1u$. Using \eqref{fte}, $\MF(\Omega)\hookrightarrow L_r(\Omega)$ is then nuclear if, and only if, $\MF(\Omega)=\ft(\Omega) \hookrightarrow L_r(\Omega)$ is nuclear which amounts to the desired result in \eqref{lim-13} by (i). We come to the source spaces $\MB(\Omega)$. The sufficiency is a consequence of \eqref{N-BT-emb} and (i) again, while the necessity follows from Theorem~\ref{comp-NE} applied to $\MAe=\MB$ and $\MAz=\mathcal{N}^0_{r,r\infty}=B^0_{r,\infty}$ together with the embeddings $L_r(\Omega)\hookrightarrow B^0_{r,\infty}(\Omega)$. 
\end{proof}


\bigskip

\noindent Dorothee D. Haroske\\
\noindent  Institute of Mathematics,
Friedrich Schiller University Jena, 07737 Jena, Germany\\
\noindent {\it E-mail}:  \texttt{dorothee.haroske@uni-jena.de}

\bigskip

\noindent Leszek Skrzypczak\\
\noindent  Faculty of Mathematics and Computer Science,
Adam Mickiewicz University,\\
 ul. Uniwersytetu Pozna\'nskiego 4, 61-614 Pozna\'n,
Poland\\
\noindent {\it E-mail}:  \texttt{leszek.skrzypczak@amu.edu.pl}

\end{document}